\documentclass{amsart}
\usepackage{amssymb}
\usepackage{amsmath}
\usepackage{amsthm}
\usepackage{amscd}
\usepackage{amstext}
\usepackage{amsfonts}
\usepackage[utf8]{inputenc}
\usepackage[all]{xy}
\usepackage[dvipdfmx]{graphicx}
\usepackage{bm}

\theoremstyle{plain}
\newtheorem{theorem}{Theorem}[section]
\newtheorem{corollary}[theorem]{Corollary}
\newtheorem{lemma}[theorem]{Lemma}
\newtheorem{proposition}[theorem]{Proposition}

\theoremstyle{definition}
\newtheorem{definition}[theorem]{Definition}
\theoremstyle{remark}
\newtheorem{remark}[theorem]{Remark}
\newtheorem{example}[theorem]{Example}

\newcommand{\N}{\mathbb{N}}
\newcommand{\Z}{\mathbb{Z}}
\newcommand{\R}{\mathbb{R}}
\newcommand{\C}{\mathbb{C}}
\newcommand{\T}{\mathbb{T}}
\newcommand{\D}{\mathbb{D}}
\newcommand{\SSS}{\mathbb{S}}
\newcommand{\K}{\mathcal{K}}

\newcommand{\A}{\mathrm{A}}
\newcommand{\AIII}{\mathrm{A\hspace{-.05em}I\hspace{-.05em}I\hspace{-.05em}I}}
\newcommand{\AI}{\mathrm{A\hspace{-.05em}I}}
\newcommand{\BDI}{\mathrm{BDI}}
\newcommand{\DD}{\mathrm{D}}
\newcommand{\DIII}{\mathrm{D\hspace{-.05em}I\hspace{-.05em}I\hspace{-.05em}I}}
\newcommand{\AII}{\mathrm{A\hspace{-.05em}I\hspace{-.05em}I}}
\newcommand{\CII}{\mathrm{C\hspace{-.05em}I\hspace{-.05em}I}}
\newcommand{\CC}{\mathrm{C}}
\newcommand{\CI}{\mathrm{CI}}

\newcommand{\I}{\mathcal{I}}
\newcommand{\cpt}{\mathrm{cpt}}
\newcommand{\pt}{\mathrm{pt}}
\newcommand{\id}{\mathrm{id}}
\newcommand{\sa}{\mathrm{sa}}
\newcommand{\ska}{\mathrm{sk}}
\newcommand{\diag}{\mathrm{diag}}

\newcommand{\CA}{$C^*$-algebra}
\newcommand{\TA}{$C^{*, \tau}$-algebra}

\DeclareMathOperator{\ind}{\mathrm{index}}
\DeclareMathOperator{\Ker}{\mathrm{Ker}}

\DeclareMathOperator{\Ad}{\mathrm{Ad}}

\newcommand{\tpi}{\pi'}

\newcommand{\HH}{\mathcal{H}}
\newcommand{\HHz}{\HH^{0}}
\newcommand{\HHi}{\HH^{\infty}}
\newcommand{\HHzi}{\HH^{0,\infty}}
\newcommand{\TT}{\mathcal{T}}
\newcommand{\TTz}{\TT^{0}}
\newcommand{\TTi}{\TT^{\infty}}
\newcommand{\TTzi}{\TT^{0,\infty}}
\newcommand{\Szi}{\mathcal{S}^{0,\infty}}
\newcommand{\sigmaz}{\sigma^0}
\newcommand{\sigmai}{\sigma^\infty}
\newcommand{\Pz}{P^0}
\newcommand{\PI}{P^{\infty}}
\newcommand{\Pzi}{P^{0,\infty}}

\newcommand{\bz}{{\bm z}}

\newcommand{\fc}{\mathfrak{c}}
\newcommand{\fC}{\mathcal{C}}
\newcommand{\fR}{\mathcal{R}}
\newcommand{\fr}{\mathfrak{r}}
\newcommand{\fJ}{\mathcal{J}}

\newcommand{\fs}{\mathfrak{s}}
\newcommand{\fe}{\mathfrak{e}}

\newcommand{\ft}{\mathfrak{t}}

\begin{document}

\title[Quarter-Plane Toeplitz Operators via Extended Symbol]{An Index Theorem for Quarter-Plane Toeplitz Operators via Extended Symbols and Gapped Invariants Related to Corner States}
\author[S. Hayashi]{Shin Hayashi}
\address{Japan Science and Technology Agency, 4-1-8 Honcho, Kawauchi, Saitama 332-0012 Japan}
\address{Advanced Institute for Materials Research, Tohoku University, 2-1-1 Katahira, Aoba, Sendai 980-8577 Japan}
\address{Mathematics for Advanced Materials - Open Innovation Laboratory, National Institute of Advanced Industrial Science and Technology, 2-1-1 Katahira, Aoba, Sendai 980-8577, Japan}
\email{{\tt shin.hayashi.e2@tohoku.ac.jp}}


\begin{abstract}
In this paper, we discuss index theory for Toeplitz operators on a discrete quarter-plane of two-variable rational matrix function symbols.
By using Gohberg--Kre{\u \i}n theory for matrix factorizations, we extend the symbols defined originally on a two-dimensional torus to some three-dimensional sphere and derive a formula to express their Fredholm indices through extended symbols.
Variants for families of (self-adjoint) Fredholm quarter-plane Toeplitz operators and those preserving real structures are also included.
For some bulk-edge gapped single-particle Hamiltonians of finite hopping range on a discrete lattice with a codimension-two right angle corner, topological invariants related to corner states are provided through extensions of bulk Hamiltonians.
\end{abstract}

\maketitle
\setcounter{tocdepth}{2}
\tableofcontents

\section{Introduction}
\label{Sect.1}
Topological corner states took much interest in condensed matter physics as a characteristic of higher-order topological insulators.
Aimed at studies of topological corner states, we discuss index theory for some Toeplitz operators on a discrete quarter-plane.
Index theory for quarter-plane Toeplitz operators has been investigated by Simonenko, Douglas--Howe \cite{Sim67, DH71}, and a necessary and sufficient condition for these operators to be Fredholm is obtained in terms of the invertibility of two associated half-plane Toeplitz operators.
Index formulas for Fredholm quarter-plane Toeplitz operators are obtained by Coburn--Douglas--Singer, Dudu\v{c}ava, Park \cite{CDS72, Dud77, Pa90}.
Coburn--Douglas--Singer derived their formula by showing that there is a deformation to some quarter-plane Toeplitz operators of a standard form preserving Fredholm indices \cite{CDS72}.
Dudu\v{c}ava employed Gohberg--Kre{\u \i}n theory for the factorization of some matrix functions on a circle \cite{GK58r, Subin67, CG81} and obtained a formula by using a construction of parametrix \cite{Dud77}.
Park obtained an index formula by a construction of a cyclic cocycle and using a pairing between $K$-theory and cyclic cohomology \cite{Pa90}.

A characteristic feature of topological insulators is the existence of topological edge states.
Although the bulk is gapped (insulating), edge states exist that account for the metallic properties of the boundary of the system.
This appearance of edge states is known to originate from a topological invariant for the gapped bulk, called the {\em bulk-edge correspondence}.
A typical example is the integer quantum Hall system in which the bulk topological invariant is known to be the first Chern number of a complex vector bundle (Bloch bundle) over a two-dimensional torus (Brillouin torus).
Bellissard investigated the quantum Hall effect through noncommutative geometry \cite{BvES94}, and Kellendonk--Richter--Schulz-Baldes gave a proof of the bulk-edge correspondence based on index theory for Toeplitz operators \cite{KRSB02}.
$K$-theory was employed for the classification of topological insulators \cite{Kit09,FM13} (see also \cite{PSB16} and the references therein).
We note that matrix factorizations is also used in recent physical studies of (first-order) topological insulators \cite{Alase}.
For {\em higher-order topological insulators} \cite{BBH17a, Schindler18}, an actively studied topic in condensed matter physics, the bulk-edge correspondence is much generalized to include corner states.
For two-dimensional second-order topological insulators, for example, the bulk and two edges whose intersection form a codimension-two corner are gapped, though there exist topological corner states, and a relation between some gapped topology and corner states are much discussed.

In \cite{Hayashi2}, a mathematical approach to topological corner states is proposed based on index theory for quarter-plane Toeplitz operators, where topological invariants for bulk-edge gapped Hamiltonians are defined as elements of a $K$-group of some \CA \ and its relation with hinge states is proved.
Although this shows a relation between some gapped topology and corner states, gapped invariants are defined abstractly and much more geometric understanding is required, in order both for investigation from the physical point of view and its computation.
For this purpose, we investigate further index theory for quarter-plane Toeplitz operators,
especially Dudu\v{c}ava's idea of using matrix factorizations \cite{Dud77} from a topological point of view.
We consider Fredholm quarter-plane Toeplitz operators of two-variable rational matrix function symbols.
For each of them, there associates two invertible half-plane Toeplitz operators having the same symbol.
In Sect.~$3$, we investigate the geometric implications of this invertibility condition.
Through the Fourier transform in a direction parallel to the boundary, a half-plane Toeplitz operator corresponds to a one-parameter family of Toeplitz operators, and the problem reduces to a study of invertible Toeplitz operators.
For an invertible Toeplitz operator of a rational matrix function symbol, Gohberg--Kre{\u \i}n theory states that there is a decomposition of the symbol as a product of two matrix-valued functions such that each factor of the decomposition can be analytically continued to a disk.
By using analytic continuation, we see that the symbol of the quarter-plane Toeplitz operator defined originally on a two-dimensional torus can be extended as a continuous nonsingular matrix-valued function over some three sphere.
This extension is shown to be independent of the choice of the factorization, therefore is canonically associated with our operator.
We then show in Sect.~$4$ that the Fredholm index of the quarter-plane Toeplitz operator is given through the three-dimensional winding number of the extended symbol (Corollary~\ref{maincorollary}).
Note that a part of its proof is based on Coburn--Douglas--Singer's idea \cite{CDS72}.
Our formula can be extended to Fredholm quarter-plane Toeplitz operators which are self-adjoint, preserving real structures and families of them, and these variants are proved in a parallel way.
In this paper, we mainly discuss families of (self-adjoint) Fredholm quarter-plane Toeplitz operators for which we use complex $K$-theory (Theorem~\ref{mainfamily}), and the results for those operators preserving real structures are contained in Sect.~$5$.
Necessary results about quarter-plane Toeplitz operators and Gohberg--Kre{\u \i}n theory for matrix factorizations used in this paper are collected in Sect.~$2$.

Applications to topological corner states are discussed in Sect.~$6$.
We consider translation invariant single-particle Hamiltonians of finite hopping range on the lattice $\Z^n$ in each of the ten Altland--Zirnbauer classes \cite{AZ97}, and discuss its restrictions onto $(\Z_{\geq 0})^2 \times \Z^{n-2}$ assuming the Dirichlet boundary condition.
When the bulk Hamiltonian and its compressions onto two half-spaces $\Z \times \Z_{\geq 0} \times \Z^{n-2}$ and $\Z_{\geq 0} \times \Z \times \Z^{n-2}$ are gapped, we associates a nonsingular matrix function (over some three sphere for two-dimensional systems) through the matrix factorization which is an extension of the bulk Hamiltonian.
We define a topological invariant for such a bulk-edge gapped Hamiltonian as a $K$-class in some topological $K$-theory group of this extended bulk Hamiltonian (Definition~\ref{def6.1}).
A relation between this gapped topological invariant and corner/hinge states is given in Theorem~\ref{thm6.2}, which provides a geometric formulation for the relation between the abstractly defined gapped topological invariant and corner states in \cite{Hayashi2}.
In order to construct extensions of bulk Hamiltonians, we need to take matrix factorizations.
For matrix factorizations of rational matrix functions on the unit circle in the complex plane, an algorithm is known \cite{GK58r, CG81, GKS03} and the finite hopping range condition is assumed correspondingly.
In \cite{Hayashi4}, a classification of topological invariants related to corner states in each of the Altland--Zirnbauer classes are proposed based on index theory, where Boersema--Loring's formulation of $KO$-theory for real \CA s \cite{BL16} is employed.
Since topological corner states are one motivation of this work, some parts of the discussions in this paper are organized in this framework.
For example, the real symmetries discussed in Sect.~\ref{Sect.5} are taken from Boersema--Loring's picture.
Note that the integration formula for our gapped topological invariants, like the integration of the Berry curvature for the first Chern number, is still missing since our three sphere is not smooth, although our formulation provides a way to understand gapped topological invariants related to corner states in a geometric way.
For example, that for a two-dimensional class AIII system is given by the three-dimensional winding number of the extension of the bulk Hamiltonian (Example~\ref{2DAIII}) and that for a three-dimensional class A system is provided as a topological invariant for an extension of the Bloch bundle (Example~\ref{3DA}).

\section{Preliminaries}
\label{Sect.2}
In this section, we collect the necessary results and notations used in this paper.

\subsection{Quarter-Plane Toeplitz operators}
\label{Sect.2.1}
Let $\T$ be the unit circle in the complex plane equipped with the normalized Haar measure.
For $f \in C(\T^n)$, we write $M_f$ for the bounded linear operator on $l^2(\Z^n)$ corresponding to the multiplication operator on $L^2(\T^n)$ generated by $f$ through the Fourier transform $L^2(\T^n) \cong l^2(\Z^n)$.
For an integer $k$, we write $\Z_{\geq k}$ for the set of integers greater than or equal to $k$.
Let ${\bm \delta_n}$ be the characteristic function of a point $n \in \Z$ and let $l^2(\Z_{\geq 0})$ be the closed subspace of $l^2(\Z)$ spanned by $\{ {\bm \delta_n} \mid n \geq 0 \}$.
Let $P$ be the orthogonal projection of $l^2(\Z)$ onto $l^2(\Z_{\geq 0})$.
For $f \in C(\T)$, the operator on $l^2(\Z_{\geq 0})$ defined by $T_f \varphi = P M_f \varphi$ for $\varphi \in l^2(\Z_{\geq 0})$ is called the {\em Toeplitz operator} of continuous symbol $f$.
Let $\TT$ be the \CA \ generated by those Toeplitz operators.
We have the following Toeplitz extension:
\begin{equation}\label{tExt}
	0 \to \K \to \TT \overset{\sigma}{\to} C(\T) \to 0,
\end{equation}
where $\sigma$ is a $*$-homomorphism that maps $T_f$ to its symbol $f$.

Let ${\bm \delta_{m,n}}$ be the characteristic function of the point $(m,n)$ in $\Z^2$, and let $\HHz$, $\HHi$ and $\HHzi$ be closed subspaces of $l^2(\Z^2)$ spanned by
$\{ {\bm \delta_{m,n}} \mid n \geq 0 \}$,
$\{ {\bm \delta_{m,n}} \mid m \geq 0 \}$, and
$\{ {\bm \delta_{m,n}} \mid m \geq 0 \ \text{and}\ n\geq 0 \}$, respectively.
Let $\Pz$, $\PI$ and $\Pzi$ be the orthogonal projection of $l^2(\Z^2)$ onto $\HHz$, $\HHi$ and $\HHzi$, respectively.
Note that $\HHzi = \HHz \cap \HHi$ and $\Pzi = \Pz \PI = \PI \Pz$.
For $f \in C(\T^2)$, the operators $T^0_f$ on $\HHz$ and $T^\infty_f$ on $\HHi$ defined by 
 $T^0_f \varphi = \Pz M_f \varphi$ for $\varphi \in \HHz$ and $ T^\infty_f \varphi= \PI M_f \varphi$ for $\varphi \in \HHi$, respectively, are called {\em half-plane Toeplitz operators}.
The operator on $\HHzi$ defined by $T^{0,\infty}_f \varphi = \Pzi M_f \varphi$ for $\varphi \in \HHzi$ is called the {\em quarter-plane Toeplitz operator}.
Let $\TTz$ and $\TTi$ be the \CA s generated by half-plane Toeplitz operators of the form $T^0_f$ and $T^\infty_f$, respectively, and let $\TTzi$ be the \CA \ generated by quarter-plane Toeplitz operators.
Note that $\TTz \cong C(\T) \bm{\otimes} \TT$ and $\TTi \cong \TT \bm{\otimes} C(\T)$ by the Fourier transform in a direction parallel to the boundary of half-planes.
Corresponding to these isomorphisms, let $\sigmaz = 1_{C(\T)} \bm{\otimes} \sigma$ and $\sigmai = \sigma \bm{\otimes} 1_{C(\T)}$ which are $*$-homomorphisms from $\TTz$ and $\TTi$ to $C(\T^2)$.
Let $\Szi$ be the pullback \CA \ of these two $*$-homomorphisms,
\begin{equation*}
\vcenter{
\xymatrix{
\Szi \ar[r]^{p^\infty} \ar[d]_{p^0} & \TT^\infty \ar[d]^{\sigma^\infty} \\
\TT^0 \ar[r]^{\sigma^0} & C(\T^2)}}
\end{equation*}
and let $\sigma_\mathcal{S} = \sigma^0 \circ p^0 = \sigma^\infty \circ p^\infty$.
The following short exact sequence of \CA s is known \cite{Pa90}:
\begin{equation}\label{tqExt}
	0 \to \K \to \TTzi \overset{\gamma}{\to} \Szi \to 0,
\end{equation}
where $\K$ is the compact operator algebra, the map $\K \to \TTzi$ is the inclusion, and $\gamma$ is a $*$-homomorphism that maps $T^{0,\infty}_f$ to $(T^0_f, T^\infty_f)$ for $f \in C(\T^2)$.

\subsection{Matrix Factorizations}
\label{Sect.2.2}
In this subsection, we collect necessary results about Gohberg--Kre{\u \i}n theory for factorizations of rational matrix functions in cases on the unit circle in the complex plane.
For details, we refer the reader to \cite{GK58r, CG81, GKS03}.

Let us consider the Riemann sphere $\hat{\C} = \C \cup \{ \infty \}$.
The unit circle $\T$ is contained in $\hat{\C}$ and $\hat{\C} \setminus \T$ consists of two connected components.
Let $D_+ = \{ z \in \C \mid \lvert z \lvert < 1\}$ and $D_- = \{ z \in \C \mid \lvert z \lvert  > 1 \} \cup \{ \infty\}$ which are open disks.
We write $\D^2$ for the closed unit disk $\T \cup D_+$ in the complex plane.
Let $f \colon \T \to GL_n(\C)$ be a nonsingular rational matrix function, that is, a nonsingular matrix function of entries consisting of rational functions with poles off $\T$.
We have the following decomposition, called a {\em (right) matrix factorization}:
\begin{equation}\label{factorization}
	f = f_- \Lambda f_+,
\end{equation}
where $f_-$, $\Lambda$ and $f_+$ are rational matrix functions on $\T$ satisfying the conditions below.
\begin{itemize}
\item $f_+$ (resp. $f_-$) admits a continuous extension onto $\T \cup D_+ = \D^2$ (resp. $\T \cup D_-$) as a nonsingular matrix function and is analytic on $D_+$ (resp. $D_-$).
We write $f_+^e$ (resp. $f_-^e$) for the extension.
\item $\Lambda$ is the diagonal matrix function of the form $\Lambda(z) = \diag(z^{\kappa_1}, \ldots, z^{\kappa_n})$ with nonincreasing sequence of integers $\kappa_1 \geq \cdots \geq \kappa_n$ called {\em partial indices}.
\end{itemize}
Among the many results known for matrix factorizations, we note the following:
\begin{itemize}
\item The Toeplitz operator $T_f$ is invertible if and only if all of the partial indices are zero. In this case, factorization (\ref{factorization}) is called a {\em canonical factorization}. 
\item If we have two canonical factorizations $f = f_- f_+ = g_- g_+$, there exists an invertible matrix $B \in GL_n(\C)$, considered a constant matrix function, satisfying $f_+ = B g_+$ and $f_- = g_- B^{-1}$.
\end{itemize}

\begin{remark}\label{rem21}
There is known a general class of Banach algebras of functions on the circle that admits matrix factorizations called decomposing Banach algebras \cite{CG81}.
One example is the Wiener algebra over the circle consisting of all complex-valued functions $f$ on $\T$ admitting an absolutely convergent Fourier series.
The results in this paper are also valid for such continuous matrix functions admitting matrix factorizations.
Note that the algebra $C(\T)$ of continuous functions is not decomposing. 
In this paper, we mainly discuss rational matrix functions on the unit circle in view of our applications discussed in Sect~\ref{Sect.6}.
\end{remark}

\section{Extension of Symbols Through Matrix Factorizations}
\label{Sect.3}
Let $f \colon \T^2 \to GL_n(\C)$ be a continuous map.
The associated quarter-plane Toeplitz operator $T^{0,\infty}_f$ is Fredholm if and only if the half-plane Toeplitz operators $T^0_f$ and $T^\infty_f$ are invertible \cite{DH71}.
In this section, we discuss the geometric implication of this condition when the symbol $f$ is a rational matrix function for both of the two variables of $\T^2$. 
By using matrix factorizations, we provide a way to extend $f$ to a nonsingular matrix-valued continuous function on a three sphere (Sect.~\ref{Sect.3.1}).
We also discuss its variants for families of them (Sect.~\ref{Sect.3.2}) and the cases when matrix functions take value in (skew-)hermitian matrices (Sect.~\ref{Sect.3.3}).

\subsection{Extension of Nonsingular Matrix Functions of Trivial Partial Indices}
\label{Sect.3.1}
Let $f$ be a nonsingular rational matrix function on the circle $\T$, and assume that the associated Toeplitz operator $T_f$ is invertible.
Let $f = f_- f_+$ be a canonical factorization of $f$.
Let $f^e_+$ and $f^e_-$ be the associated extensions of $f_+$ and $f_-$ onto $\T \cup D_+$ and $\T \cup D_-$, respectively.
For two disks $\D^2 = \T \cup D_+$ and $\T \cup D_-$, we consider the following identification:
\begin{equation*}
	I \colon \T \cup D_+ \to \T \cup D_-, \ I(z) = \bar{z}^{-1},
\end{equation*}
where we set $I(0) = \infty$.
By using this identification, we associate the following nonsingular matrix-valued continuous map for a canonical factorization:
\begin{equation}\label{mainext1}
	f^e = (I^*f_-^e) \cdot f_+^e \colon \D^2 \to GL_n(\C).
\end{equation}
Explicitly, $f^e(z) = f_-^e(\bar{z}^{-1}) f_+^e(z)$ for $z \in \D^2$.
Since $\bar{z}^{-1} = z$ for $z \in \T$, $f^e$ coincides with $f$ on $\T$ and is its continuous extension onto $\D^2$.

\begin{lemma}\label{extension1}
	$f^e$ is independent of the choice of canonical factorization.
\end{lemma}
\begin{proof}
For two canonical factorizations $f = f_- f_+ = g_- g_+$, there exists $B \in GL_n(\C)$ satisfying $f_+ = B g_+$ and $f_- = g_- B^{-1}$.
Note that $f^e_+ = (B g_+)^e = B g^e_+$ and $f_-^e = g_-^e B^{-1}$, which follow from the uniqueness of analytic continuation.
Therefore, for $z \in \D^2$,
\begin{equation*}
	f_-^e(\bar{z}^{-1}) \cdot f_+^e(z) = g_-^e(\bar{z}^{-1})B^{-1} \cdot B g_+^e(z) = g_-^e(\bar{z}^{-1}) \cdot g_+^e(z).
\end{equation*}
\end{proof}
By Lemma~\ref{extension1}, for a rational matrix function $f \colon \T \to GL_n(\C)$ whose associated Toeplitz operator $T_f$ is invertible, there is a canonically associated extension $f^e$ onto the disk $\D^2$.
For a non-negative real number $t$, we write $\T_t = \{ z \in \C \mid \lvert z \lvert \ =t \}$.
For $0 \leq t \leq 1$, let $m_t \colon \T \to \T_t$ be the map defined by $m_t(z)=tz$.
We take the pullback $m_t^* (f^e\lvert_{\T_t})$ of $f^e\lvert_{\T_t}$ onto $\T$ by $m_t$.
Let us consider the Toeplitz operator $T_t := T_{m_t^* (f^e\lvert_{\T_t})}$ associated with this matrix function.
In other words, we consider $\D^2$ to be a family of circles of radius $0 \leq t \leq 1$ and consider the associated family of Toeplitz operators.
For these operators, the following holds:

\begin{lemma}\label{extension3}
Let $f$ be a nonsingular rational matrix function on $\T$ of trivial partial indices, and consider its extension $f^e \colon \D^2 \to GL_n(\C)$ in (\ref{mainext1}).
For $0 \leq t \leq 1$, the Toeplitz operator $T_t$ in is invertible.
\end{lemma}
\begin{proof}
The invertibility of $T_1 = T_f$ is included in our assumption.
When $t=0$, the Toeplitz operator $T_0$ is associated with nonsingular constant matrix function and is invertible.
Let us consider the case in which $0 < t <1$.
For $z \in \T$, we have
\begin{equation*}
	m_t^* (f^e\lvert_{\T_t})(z) = f^e(tz) = f^e_-((\overline{tz})^{-1}) f^e_+(tz) = f^e_-(t^{-1}z) f^e_+(tz).
\end{equation*}
Therefore, we have the following decomposition for the symbol of $T_t$:
\begin{equation}\label{eqnpf1}
	m_t^* (f^e\lvert_{\T_t})= m^*_{t^{-1}} (f^e_-\lvert_{\T_{t^{-1}}}) \cdot m_t^*(f^e_+\lvert_{\T_t}).
\end{equation}
Each component of equation (\ref{eqnpf1}) is a rational matrix function on $\T$.
Let $D_{+,t} = \{ z \in \C \mid \lvert z \lvert \ < t\}$ and $D_{-,{t^{-1}}} = \{ z \in \C \mid \lvert z \lvert \ > t^{-1} \} \cup \{ \infty\}$.
Let us consider the maps $m_{+,t} \colon D_+ \to D_{+,t}$, $z \mapsto tz$, and $m_{-,t^{-1}} \colon D_- \to D_{-,t^{-1}}$, $z \mapsto t^{-1}z$.
Since $f^e_-$ and $f^e_+$ are extensions of $f_-$ and $f_+$ that are analytic on $D_-$ and $D_+$, pullbacks of their restrictions $m_{-,t^{-1}}^* (f^e_-\lvert_{D_{-,{t^{-1}}}})$ and $m_{+,t}^*(f^e_+\lvert_{D_{+,t}})$ onto $D_-$ and $D_+$ also provide such extensions of $m^*_{t^{-1}} (f^e_-\lvert_{\T_{t^{-1}}})$ and $m_t^*(f^e_+\lvert_{\T_t})$.
Therefore, equation (\ref{eqnpf1}) is a canonical factorization of $m_t^* (f^e\lvert_{\T_t})$, and the associated Toeplitz operator $T_t$ is invertible.
\end{proof}

\subsection{Extension of Families of Matrix Functions}
\label{Sect.3.2}
We next extend the discussions in Sect.~\ref{Sect.3.1} to families of rational matrix functions of trivial partial indices.
Matrix factorizations for such families are studied by \v{S}ubin \cite{Subin67}.

\begin{lemma}\label{extension2}
	Let $X$ be a topological space. Let $f \colon \T \times X \to GL_n(\C)$ be a continuous map such that for each $x \in X$, $f(x)$ is a rational matrix function on $\T$ of trivial partial indices.
Through the matrix factorization, there canonically associates a continuous map $f^e \colon \D^2 \times X \to GL_n(\C)$ that extends $f$.
\end{lemma}
\begin{proof}
Following \cite{Subin67}, for each $x_0 \in X$, there exists an open neighborhood $U \subset X$ of $x_0$ and continuous matrix functions $f_+$ and $f_-$ on $\T \times U$ such that for $x \in U$, $f(x) = f_-(x) f_+(x)$ is a canonical factorization.
By using this factorization, we obtain a continuous extension $f^e \colon \D^2 \times U \to GL_n(\C)$ of $f$ as in equation (\ref{mainext1}).
As in Lemma~\ref{extension1}, this $f^e$ is independent of the choice of canonical factorization.
We cover $X$ by such open sets $\{ U_\alpha \}_{\alpha \in J}$.
When $U_\alpha \cap U_\beta \neq \emptyset$, we may consider two extensions of $f\lvert_{\T \times (U_\alpha \cap U_\beta)}$ onto $\D^2 \times (U_\alpha \cap U_\beta)$ corresponding to extensions onto $\D^2 \times U_\alpha$ and $\D^2 \times U_\beta$, and they coincide by Lemma~\ref{extension1}.
Therefore, we obtain the desired extension onto $\D^2 \times X$.
\end{proof}

We next consider families of two-variable rational matrix functions whose associated quarter-plane Toeplitz operators are Fredholm.
Let
\begin{equation*}
	\tilde{\SSS}^3 := \partial (\D^2 \times \D^2) = \T \times \D^2 \underset{\T^2}{\cup} \D^2 \times \T,
	\vspace{-1mm}
\end{equation*}
which is topologically a three-dimensional sphere.
Since $\D^2 \subset \C$, we consider $\tilde{\SSS}^3$ as a subspace of $\C^2$ and use complex variables $(z,w) \in \C^2$ to parametrize $\tilde{\SSS}^3$.

\begin{proposition}\label{ext3}
	Let $X$ be a topological space. Let $f \colon \T^2 \times X \to GL_n(\C)$
be a continuous map such that for each $x \in X$, $f(x)$ is a two-variable rational matrix function for which the associated quarter-plane Toeplitz operator $T^{0,\infty}_{f(x)}$ is Fredholm.
Through matrix factorization, there canonically associates a continuous map $f^E \colon \tilde{\SSS}^3 \times X \to GL_n(\C)$ that extends $f$.
\end{proposition}
\begin{proof}
For each $x \in X$, since $T^{0,\infty}_{f(x)}$ is Fredholm, both of the associated half-plane Toeplitz operators $T^0_{f(x)}$ and $T^\infty_{f(x)}$ are invertible \cite{DH71}.
Through a Fourier transform in a direction parallel to the boundary, the invertible half-plane Toeplitz operator $T^0_{f(x)}$ corresponds to a family of invertible Toeplitz operators $\{ T_{f(z, \cdot, x)} \}_{z \in \T}$ parametrized by the circle.
Therefore, by Lemma~\ref{extension2}, there canonically associates a continuous extension
$f^e \colon \T \times \D^2 \times X \to GL_n(\C)$
of $f$ through matrix factorization.
By the invertibility of $T^\infty_{f(x)}$, we also obtain an extension of $f$ onto $\D^2 \times \T \times X$ through matrix factorization.
Combined with them, we obtain an extension $f^E$ of $f$ as a nonsingular matrix-valued continuous function on $\tilde{\SSS}^3 \times X$.
\end{proof}

\subsection{Hermitian and Skew-Hermitian Matrix Functions}
\label{Sect.3.3}
In this subsection, we discuss the case in which the nonsingular rational matrix functions in Sect.~\ref{Sect.3.1} and \ref{Sect.3.2} take values in {\em hermitian} or {\em skew-hermitian} matrices.
Let $GL_n(\C)^\sa$ (resp. $GL_n(\C)^\ska$) be the space of $n$-by-$n$ hermitian (resp. skew-hermitian) invertible matrices.

\begin{lemma}\label{ffh}
Let $\mathfrak{I}$ be $GL_n(\C)^\sa$ or $GL_n(\C)^\ska$.
Let $X$ be a topological space and $f \colon \T \times X \to \mathfrak{I}$ be a continuous map such that, for each $x \in X$, $f(x)$ is a rational matrix function on $\T$ with trivial partial indices.
Then the extension $f^e$ of $f$ in Lemma~\ref{extension2} is also a hermitian or skew-hermitian matrix function; that is, $f^e \colon \D^2 \times X \to \mathfrak{I}$.
\end{lemma}

\begin{proof}
We consider hermitian matrix functions, that is, the case in which $\mathfrak{I} = GL_n(\C)^\sa$.
By Lemma~\ref{extension2}, it is sufficient to show that $f^e(x)$ for each $x \in X$ is a hermitian matrix-valued function.
Therefore, it is sufficient to consider the case when $X$ is one point set and we assume this condition.
Let $f = f_- f_+$ be a canonical factorization of $f$.
Since $f$ is hermitian, we have 
$f_- f_+ = f = f^* = f_+^* f_-^*$.
The equation $f = f_+^* f_-^*$ is also a canonical factorization since $(f^e_+)^* \circ I^{-1}$ (resp. $(f^e_-)^* \circ I$) provides a continuous extension of $f_+^*$ onto $\T \cup D_-$ (resp. $f_-^*$ onto $\T \cup D_+$), which is analytic on $D_-$ (resp. $D_+$).
Therefore, there exists $B \in GL_n(\C)$ such that $f_+^* = f_- B^{-1}$ and $f_-^* = B f_+$.
The uniqueness of analytic continuation leads to the relations
$(f^e_+)^* \circ I^{-1} = (f_+^*)^e = f_-^e B^{-1}$ and
$(f^e_-)^* \circ I = (f_-^*)^e = B f_+^e$.
Therefore, for $z \in \D^2$,
\begin{equation*}
	(f^e(z))^* = (f^e_-(\bar{z}^{-1}) \cdot f_+^e(z))^* 
	  = (f_+^e(z))^* (f^e_-(\bar{z}^{-1}))^*
	  = f_-^e(\bar{z}^{-1})B^{-1} B f_+^e(z)
	  = f^e(z).
\end{equation*}

The result for skew-hermitian matrix functions is proved in a similar manner.
\end{proof}

By Proposition~\ref{ext3} and Lemma~\ref{ffh}, we obtain the following result:
\begin{proposition}\label{exthf}
Let $\mathfrak{I}$ be $GL_n(\C)^\sa$ or $GL_n(\C)^\ska$.
Let $X$ be a topological space. Let $f \colon \T^2 \times X \to \mathfrak{I}$
be a continuous map such that, for each $x \in X$, $f(x)$ is a two-variable rational matrix function for which the associated quarter-plane Toeplitz operator $T^{0,\infty}_{f(x)}$ is Fredholm.
Through matrix factorization, there canonically associates a continuous map
$f^E \colon \tilde{\SSS}^3 \times X \to \mathfrak{I}$
that extends $f$.
\end{proposition}

\section{Index Theorem for Quarter-Plane Toeplitz Operators via Extended Symbols}
\label{Sect.4}
In this section, we give a formula to express family indices for (self-adjoint) Fredholm quarter-plane Toeplitz operators of two-variable rational matrix function symbols by using extended symbols obtained through the matrix factorizations in Sect.~\ref{Sect.3}.
The main theorem in this section is Theorem~\ref{mainfamily} which is formulated by using complex $K$-theory and we start from preliminaries of $K$-theory.

Let $n_0=2$ and $n_1 = 1$.
Let $M^{(0)}_n(\C) = M_{2n}(\C)^\sa$, that is, the set of $2n$-by-$2n$ hermitian matrices, and $M^{(1)}_n(\C) = M_n(\C)$.
For $i =0,1$, let $GL^{(i)}_n(\C)$ be the subspace of $M^{(i)}_n(\C)$ consisting of invertible matrices, therefore, $GL^{(0)}_n(\C) = GL_{2n}(\C)^\sa$ and $GL^{(1)}_n(\C) = GL_n(\C)$.
Let $I^{(0)} = \diag(1, -1)$ and $I^{(1)} = 1$.
We write $I^{(i)}_m$ for the $n_i m$-by-$n_i m$ diagonal matrix $\diag(I^{(i)}, \ldots, I^{(i)})$.
For a unital \CA \ $A$,
let $GL^{(0)}_n(A)$ be the set of self-adjoint invertible elements in $M_{2n}(A)$, and let $GL^{(1)}_n(A)$ be the set of invertible elements in $M_n(A)$.
For $i=0,1$, let $U^{(i)}_n(A)$ be the subspace of $GL^{(i)}_n(A)$ consisting of unitary elements.
The $K$-group $K_i(A)$ for $i=0,1$ is defined as $K_i(A) = \cup_{n=1}^\infty GL^{(i)}_n(A)/\sim_i$, where the equivalence relation $\sim_i$ is generated by homotopy and stabilization by $I^{(i)}$.
For a compact Hausdorff space $X$, complex topological $K$-groups are defined as $K^{-i}(X) = K_i(C(X))$.
For a locally compact Hausdorff space $Y$, we denote $K^{-i}_\cpt(Y)$ for the compactly supported $K$-group of $Y$.
For the basics of $K$-theory used in this paper, we refer the reader to \cite{At67, At68, AS69, HR00, RLL00, BL16}.

Let $i = 0$ or $1$, and let $X$ be a compact Hausdorff space.
Let $f \colon \T^2 \times X \to GL^{(i)}_N(\C)$ be a continuous map, such that for each $x \in X$, $f(x)$ is a two-variable rational matrix function, and the associated quarter-plane Toeplitz operator $T^{0,\infty}_{f(x)}$ is Fredholm.
When $i=1$, $\{ T^{0,\infty}_{f(x)} \}_{x \in X}$ is a family of Fredholm operators which defines an element of the even complex $K$-group $K^0(X)$.
When $i=0$, $\{ T^{0,\infty}_{f(x)} \}_{x \in X}$ is a family of self-adjoint Fredholm operators.
In this case, we may stabilize $f$ if necessary (i.e., take a direct sum with $I^{(0)}$) and assume that the essential spectrum of $T^{0,\infty}_{f(x)}$ is not contained in the set of positive real numbers $\R_{>0}$ or the set of negative real numbers $\R_{<0}$ for any $x \in X$.
Under this assumption, $\{ T^{0,\infty}_{f(x)} \}_{x \in X}$ defines an element of the odd complex $K$-group $K^{1}(X)$.
In both of these cases, we write $[T^{0,\infty}_f]$ for the $K$-class of the family of (self-adjoint) Fredholm\footnote{Self-adjoint Fredholm operators when $i=0$ and Fredholm operators when $i=1$. In this section, the term {\em (self-adjoint)} should be read when $i = 0$.} quarter-plane Toeplitz operators.
By using matrix factorization, there is an associated continuous extension $f^E$ of $f$ onto $\tilde{\SSS}^3 \times X$ that takes values in hermitian invertible matrices (when $i=0$; see Proposition~\ref{exthf}) or invertible matrices (when $i=1$; see Proposition~\ref{ext3}).
This matrix function $f^E$ defines an element $[f^E]$ of the $K$-group $K^0(\tilde{\SSS}^3 \times X)$ when $i=0$, or $K^{-1}(\tilde{\SSS}^3 \times X)$ when $i=1$.
Let us consider the following isomorphism,
\begin{equation}\label{Kdecomp}
	K^{-i}(\tilde{\SSS}^3 \times X) \cong K^{-i}(X) \oplus K^{-i}_\cpt(\R^3 \times X) \cong K^{-i}(X) \oplus K^{-i-3}(X).
\end{equation}
We write $\beta \colon K^{-i}(\tilde{\SSS}^3 \times X) \to K^{-i+1}(X)$ for the composite of the projection
$K^{-i}(\tilde{\SSS}^3 \times X) \to K^{-i-3}(X)$ through the above decomposition and the Bott periodicity isomorphism $K^{-i-3}(X) \cong K^{-i+1}(X)$.
The following is the main theorem in this paper.

\begin{theorem}\label{mainfamily}
Let $X$ be a compact Hausdorff space.
Let $f \colon \T^2 \times X \to GL^{(i)}_N(\C)$ be a continuous map such that for each $x \in X$, $f(x)$ is a two-variable rational matrix function, and the associated quarter-plane Toeplitz operator $T^{0,\infty}_{f(x)}$ is Fredholm.
Let $f^E \colon \tilde{\SSS}^3 \times X \to GL^{(i)}_N(\C)$ be the extension of $f$ through matrix factorization in the hermitian case of Proposition~\ref{exthf} (when $i=0$) or Proposition~\ref{ext3} (when $i=1$).
Then, $ [T^{0,\infty}_f] = \beta([f^E])$ in the $K$-group $K^{-i+1}(X)$.
\end{theorem}

For the $f$ in Theorem~\ref{mainfamily}, the associated family of half-plane Toeplitz operators $\{ T^0_{f(x)}\}_{x \in X}$ and $\{ T^\infty_{f(x)}\}_{x \in X}$ are (self-adjoint) invertible and define an element $[(T^0_f, T^\infty_f)]$ of the $K$-group $K_i(\Szi \bm{\otimes} C(X))$.
Let
\begin{equation*}
\partial^{\text{qT}} \colon K_i(\Szi \bm{\otimes} C(X)) \to K_{i-1}(C(X)) \cong K^{-i+1}(X)
\end{equation*}
be the boundary map of the six-term exact sequence for $K$-theory of \CA s associated with the following extension obtained by taking a tensor product of the sequence (\ref{tqExt}) with $C(X)$,
\begin{equation}\label{exttens1}
	0 \to \K \bm{\otimes} C(X) \to \TTzi \bm{\otimes} C(X) \overset{\gamma \otimes 1}{\longrightarrow} \Szi \bm{\otimes} C(X) \to 0.
\end{equation}
Note that $\partial^{\text{qT}}([(T^0_f, T^\infty_f)]) = [T^{0,\infty}_f]$.
Let us consider the following diagram:
\begin{equation}\label{diagpsi}
\vcenter{
\xymatrix{
K^{-i}(\T \times X) \oplus K^{-i}(\T \times X) \ar[d]_{\partial^{\text{pair}}} \ar[rd]^{\hspace{3mm} \partial^{\text{T}} \oplus - \partial^{\text{T}}}& \\
K^{-i+1}([-2,2] \times \T \times X, \{ \pm 2 \} \times \T \times X) \ar[r]^{\hspace{2cm}\alpha} & K^{-i+1}(X)
}}
\end{equation}
where $\partial^{\text{T}}$ is the boundary map associated with the Toeplitz extension (\ref{tExt}) and $\partial^{\text{pair}}$ is the boundary map of the six-term exact sequence of topological $K$-theory associated with the pair $([-2,2] \times \T \times X, \{ \pm 2 \} \times \T \times X)$.
Maps $\partial^{\text{T}} \oplus - \partial^{\text{T}}$ and $\partial^{\text{pair}}$ are surjective, and
the kernel of $\partial^{\text{pair}}$ is the diagonals that are contained in the kernel of $\partial^{\text{T}} \oplus - \partial^{\text{T}}$.
Therefore, the horizontal map $\alpha$ making the diagram commutative is induced.
To show Theorem~\ref{mainfamily}, we will construct the following key diagram:
\begin{equation}
\label{maindiagfamily}
\vcenter{
\xymatrix{
K_i(\Szi \bm{\otimes} C(X)) \ar[d]_{\psi} \ar[rrd]^{\partial^{\text{qT}}} & &\\
K^{-i+1}([-2,2] \times \T \times X, \{ \pm 2 \} \times \T \times X) \bigl/\Ker(\alpha) \ar[rr]^{\hspace{2cm}\bar{\alpha}} & & K^{-i+1}(X) \\
K^{-i}(\tilde{\SSS}^3 \times X) \ar[u]^{\phi} \ar[rru]_{\beta} & &
}}
\end{equation}
The map $\alpha$ is surjective and the induced homomorphism $\bar{\alpha}$ in the above diagram is an isomorphism.
Note that there is the following isomorphism,
\begin{equation*}
	K^{-i+1}([-2,2] \times \T\times X, \{ \pm 2 \} \times \T \times X)
		\cong K^{-i}(\T \times X)
		\cong K^{-i}(X) \oplus K^{-i-1}(X).
\end{equation*}
Through this decomposition, $\Ker(\alpha) \cong K^{-i}(X)$ and we have the isomorphism,
\begin{equation}\label{identifyK}
	K^{-i+1}([-2,2] \times \T\times X, \{ \pm 2 \} \times \T \times X) / \Ker(\alpha) \cong K^{-i-1}(X).
\end{equation}
Through this isomorphism, the map $\bar{\alpha}$ correspond to the Bott periodicity isomorphism $K^{-i-1}(X) \cong K^{-i+1}(X)$.

\subsection{Construction of the Map $\psi$}
\label{Sect4.1}
In this subsection, we construct the homomorphism $\psi$ in the diagram (\ref{maindiagfamily}).
For its construction, Coburn--Douglas--Singer's idea to use homotopy lifting property for some fibrations plays a key role, and we introduce them first.

Let $U_n(\TT)$ and $U_n(C(\T))$ be subspaces of $M_n(\TT)$ and $M_n(C(\T))$, respectively, consisting of unitary elements,
let $U_n(\TT)^\sa$ and $U_n(C(\T))^{\sa}$ be subspaces of self-adjoint unitaries.
The map $\sigma$ in (\ref{tExt}) induces the following maps, which we also denote as $\sigma$:
\begin{equation}\label{fib1}
	\sigma \colon U_n(\TT) \to U_n(C(\T)).
\end{equation}
\begin{equation}\label{fib2}
	\sigma \colon U_n(\TT)^{\sa} \to U_n(C(\T))^{\sa}.
\end{equation}
The map (\ref{fib1}) is a Hurewicz fibration, which is used in \cite{CDS72}.
We also have its variants (Lemma~\ref{fiblem} and Proposition~\ref{refib}) which will be well-known \cite{Wood, AS69}, though we briefly contain its proof since, as in \cite{CDS72}, they will play a key role in our discussion.
The proof of Lemma~\ref{fiblem} is simply an application of discussions around Proposition~$4.1$ of \cite{AS69} to the Toeplitz extension.
For a space $Y$ and its element $y$, we write $Y_y$ for the connected component of $Y$ containing $y$.
For a Banach algebra $A$, its subset $S \subset A$ and an $\R$-linear operator $\fs$ on $A$ of order two, we write $S^\fs$ and $S^{-\fs}$ for the subsets of $S$ that are pointwise fixed by $\fs$ and $-\fs$, respectively.

\begin{lemma}\label{fiblem}
	The map (\ref{fib2}) is a Hurewicz fibration.
\end{lemma}
\begin{proof}
Since $U_n(C(\T))^{\sa}$ is a paracompact Hausdorff space, it is sufficient to show that (\ref{fib2}) is a fiber bundle \cite{Spanier66}.
Let $u$ be a self-adjoint unitary in $M_n(C(\T))$.
There exists a self-adjoint invertible lift $x \in M_n(\TT)$ of $u$ since a self-adjoint Fredholm operator can be perturbed to a self-adjoint invertible operator by a self-adjoint compact operator.
Then, the self-adjoint unitary $x\lvert x \lvert^{-1} \in M_n(\TT)$ satisfies $\sigma(x \lvert x\lvert^{-1}) = u$, and the map $\sigma \colon U_n(\TT)^{\sa} \to U_n(C(\T))^{\sa}$ is surjective.
It is now sufficient to show that, for any $s \in U_n(\TT)^\sa$, the map $\sigma \colon U_n(\TT)^{\sa}_s \to U_n(C(\T))^{\sa}_{\sigma(s)}$ is a fiber bundle.
Note that the map $\sigma \colon M_n(\TT) \to M_n(C(\T))$ has a continuous linear section given by compression, that is, mapping $f$ to $T_f$.
Let $k = s - T_{\sigma(s)}$ which is a self-adjoint compact operator, and let $l \colon M_n(C(\T)) \to M_n(\TT)$ be a map given by $l(f) = T_f +k$.
This map $l$ preserves self-adjoint elements and satisfies $l(\sigma(s)) = s$.
Combined with the map
$d \colon GL_n(\TT) \to U_n(\TT)$, $d(T) = T\lvert T\lvert^{-1}$,
we obtain a continuous local section of the map $\sigma \colon U_n(\TT)^{\sa}_s \to U_n(C(\T))^{\sa}_{\sigma(s)}$ in a neighborhood of $\sigma(s)$ that maps $\sigma(s)$ to $s$.
Let $\fs$ be an operator on $M_n(\TT)$ given by the conjugation of $s$.
Then, as in \cite{Wood}, the quotient map $GL_n(\TT)_1 \to (GL_n(\TT) / (GL_n(\TT))^\fs)_{1}$ has a continuous local section which follows from \cite{Michael59}.
Combined with the map $d$,
we obtain a continuous local section of the quotient map $U_n(\TT)_1 \to (U_n(\TT)/U_n(\TT)^\fs)_{1}$.
Let us consider the map
$U_n(\TT) \to U_n(\TT)^\sa$ given by $T \mapsto T s T^*$,
that comes from the action of unitaries to self-adjoint unitaries by conjugation.
Its stabilizer subgroup at $s$ is $U_n(\TT)^\fs$.
As in Lemma~$4.1$ of \cite{Wood}, each orbit of this action is open and induces a homeomorphism
$(U_n(\TT)/U_n(\TT)^\fs)_{1} \overset{\cong}{\longrightarrow} U_n(\TT)^\sa_s$.
Let $\ft$ be an operator on $M_n(C(\T))$ given by the conjugation of $\sigma(s)$.
We consider similar discussion for the algebra $C(\T)$ and obtain the following diagram,
\[\xymatrix{
 & U_n(\TT)^\sa_s \ar[r]^\sigma & U_n(C(\T))^\sa_{\sigma(s)} \ar@/_13pt/[l]_-{\text{loc. sect.}}  \\
U_n(\TT)_1 \ar[r] \ar@<-1ex>@{^{(}->}[d]& (U_n(\TT)/U_n(\TT)^\fs)_{1} \ar@/_13pt/[l]_-{\text{loc. sect.}} \ar[r]^{\bar{\sigma} \hspace{0.5cm}} \ar@{^{(}->}[d] \ar[u]^\cong_{\text{conj.}} & (U_n(C(\T))/U_n(C(\T))^{\ft})_{1} \ar[u]_\cong^{\text{conj.}} \\
GL_n(\TT)_1 \ar[r] \ar@<-1ex>[u]_-d & (GL_n(\TT) / (GL_n(\TT))^\fs)_{1} \ar@/^13pt/[l]^-{\text{loc. sect.}} & &
}\]
where $\bar{\sigma}$ is the map induced by $\sigma$.
Therefore, $U_n(\TT)_1 \to (U_n(C(\T))/U_n(C(\T))^{\ft})_{1}$ has a continuous local section and the result follows (see Sect.~$7.4$ of \cite{Steenrod51}).
\end{proof}

Let $i = 0$ or $1$, and let $(T^0, T^\infty)$ be an element in $GL^{(i)}_N(\Szi \bm{\otimes} C(X))$ for some $N \in \N$, and let $[(T^0, T^\infty)] \in K_i(\Szi \bm{\otimes} C(X))$ be its $K$-class.
Let $f = (\sigmaz \bm{\otimes} 1_{C(X)})(T^0) = (\sigmai \bm{\otimes} 1_{C(X)})(T^\infty)$, which is an element in $C(\T^2 \times X, GL^{(i)}_N(\C))$.
Let $\pi \colon \T^2 \times X \to X$ be the projection.

\begin{lemma}\label{lemma2}
For a sufficiently large integer $N' \geq N$, there exists a continuous map $g \colon X \to GL^{(i)}_{N'}(\C)$ such that $f \oplus I^{(i)}_{N'-N}$ is homotopic to $\pi^* g$ in $C(\T^2 \times X, GL^{(i)}_{N'}(\C))$.
We take a path $\{ g_t \}_{0 \leq t \leq 1}$ from $g_0 = f \oplus I^{(i)}_{N'-N}$ to $g_1 = \pi^* g$ in $C(\T^2 \times X, GL^{(i)}_{N'}(\C))$.
Then, there also exists a path $\{ (T^0_t,T^\infty_t) \}_{0 \leq t \leq 1}$ in $GL^{(i)}_{N'}(\Szi \bm{\otimes} C(X))$ satisfying the following conditions:
\begin{enumerate}
\renewcommand{\labelenumi}{(\roman{enumi})}
\item $(T^0_0, T^\infty_0) = (T^0 \oplus I^{(i)}_{N'-N}, T^\infty \oplus I^{(i)}_{N'-N})$.
\item $(\sigmaz \bm{\otimes} 1_{C(X)})(T^0_t) = (\sigmai \bm{\otimes} 1_{C(X)})(T^\infty_t) = g_t$ for $0 \leq t \leq 1$.
\item There exist $K \in \N$ and a continuous map $t^0 \colon \T \times X \to GL^{(i)}_{KN'}(\C)$ such that for $(z, x) \in \T \times X$,
$T^0_1(z,x) = t^0(z,x) \oplus g(x)$ as an operator on $l^2(\Z_{\geq 0}; \C^{n_iN'}) \cong \C^{n_iKN'} \oplus l^2(\Z_{\geq K}; \C^{n_iN'})$.
\item There exist $L \in \N$ and a continuous map $t^\infty \colon \T \times X \to GL^{(i)}_{LN'}(\C)$ such that for $(w, x) \in \T \times X$,
$T^\infty_1(w,x) = t^\infty(w,x) \oplus g(x)$ as an operator on $l^2(\Z_{\geq 0}; \C^{n_iN'}) \cong \C^{n_iLN'} \oplus l^2(\Z_{\geq L}; \C^{n_iN'})$.
\end{enumerate}
In $\rm(iii)$, we consider $T^0_1 \in GL^{(i)}_{N'}(\TTz \bm{\otimes} C(X)) \cong GL^{(i)}_{N'}(C(\T) \bm{\otimes} \TT \bm{\otimes} C(X))$ as a family of Toeplitz operators parametrized by $\T \times X$, and similarly for $\rm(iv)$.
\end{lemma}

\begin{proof}
We first take the deformation retraction of invertible elements $T^0$ and $T^\infty$ to unitaries; that is, for $0 \leq t \leq \frac{1}{3}$ and $j \in \{0, \infty\}$, let
$Q^j_t = (1-3t)T^j + 3t T^j \lvert T^j \lvert^{-1}$
and $f_t = (\sigma^0 \bm{\otimes} 1_{C(X)})(Q^0_t) = (\sigma^\infty \bm{\otimes} 1_{C(X)})(Q^\infty_t)$.
We next consider the following two homomorphisms,
\begin{equation*}
	(\sigma_\mathcal{S} \bm{\otimes} 1_{C(X)})_* \colon K_i(\Szi \bm{\otimes} C(X)) \to K_i(C(\T^2 \times X)) \cong K^{-i}(\T^2 \times X),
\end{equation*}
\begin{equation}\label{piast}
	\pi^* \colon K^{-i}(X) \to K^{-i}(\T^2 \times X).
\end{equation}
Note that $K_i(\Szi) = \Z$ for $i=0,1$, and that the map $(\sigma_\mathcal{S} \bm{\otimes} 1_{C(X)})_*$ is computed by using the K\"unneth theorem \cite{Sc82,Pa90}.
The images of the above two maps coincide.
Since
\begin{equation*}
	[f] = (\sigma_\mathcal{S} \bm{\otimes} 1_{C(X)})_*([T^0,T^\infty]) \in \mathrm{Im}((\sigma_\mathcal{S} \bm{\otimes} 1_{C(X)})_*) = \mathrm{Im}(\pi^*),
\end{equation*}
there exist $g \colon X \to U^{(i)}_{N'}(\C)$ for some $N'$ satisfying
$[f] = [f_{\frac{1}{3}}] = \pi^*[g]$ in $K^{-i}(\T^2 \times X)$.
Note that $\pi^*$ is injective, and the $K$-class $[g]$ of $g$ is uniquely determined.
Since the equivalence relation used to define the $K$-group $K^{-i}(\T^2 \times X)$ is generated by homotopy and stabilization, if we take $N'$ to be sufficiently large, there also exists a homotopy in $C(\T^2 \times X, U^{(i)}_{N'}(\C))$ from $f_{\frac{1}{3}} \oplus I^{(i)}_{N'-N}$ to $\pi^* g$.
Let $T^0_t = Q^0_t \oplus I^{(i)}_{N'-N}$, $T^\infty_t = Q^\infty_t \oplus I^{(i)}_{N'-N}$ and $g_t = f_t \oplus I^{(i)}_{N'-N}$ for $0 \leq t \leq \frac{1}{3}$, and for $\frac{1}{3} \leq t \leq \frac{2}{3}$, let $g_t$ be this homotopy from $f_{\frac{1}{3}} \oplus I^{(i)}_{N'-N}$ to $\pi^* g$.
By the homotopy lifting property of fibration (\ref{fib1}), there exists a homotopy $T^0_t \colon \T \times X \to U^{(i)}_{N'}(\TT)$ for $\frac{1}{3} \leq t \leq \frac{2}{3}$ that lifts $g_t$ starting from $T^0_{\frac{1}{3}} = Q^0_{\frac{1}{3}} \oplus I^{(i)}_{N'-N}$.
Note that $\sigma^0(T^0_{\frac{2}{3}}(x))(z) = g_{\frac{2}{3}}(z,x) = g(x)$ is independent of $z \in \T$.
By a similar discussion for $T^\infty$, we obtain a homotopy $T^\infty_t \colon \T \times X \to U^{(i)}_{N'}(\TT)$ for $\frac{1}{3} \leq t \leq \frac{2}{3}$ starting from $T^\infty_{\frac{1}{3}} = Q^\infty_{\frac{1}{3}} \oplus I^{(i)}_{N'-N}$ and lifts $g_t$.
Since $\sigma^0(T^0_t(x)) = \sigma^\infty(T^\infty_t(x))$ and $\sigma^0(T^0_1(x))(z) = \sigma^\infty(T^\infty_1(x))(w) = g(x)$, we obtain the following:
\begin{itemize}
\item	$(T^0_t, T^\infty_t) \in GL^{(i)}_{N'}(\Szi \bm{\otimes} C(X))$ for $0 \leq t \leq \frac{2}{3}$,
\item	$T^0_\frac{2}{3}(z,x) = g(x) + k^0(z,x)$, where $k^0(z,x)$ is a compact operator,
\item	$T^\infty_\frac{2}{3}(w,x) = g(x) + k^\infty(w,x)$, where $k^\infty(w,x)$ is a compact operator.
\end{itemize}
For $\frac{2}{3} \leq t \leq 1$, let $g_t = g_{\frac{2}{3}}$.
We perturb $k^0$ and $k^\infty$ to families of (self-adjoint) finite rank operators to obtain homotopies $T^0_t$ and $T^\infty_t$ for $\frac{2}{3} \leq t \leq 1$ in the space of (self-adjoint) invertible half-plane Toeplitz operators and obtain the desired result.
\end{proof}

For simplicity, in what follows, we assume that our representative $(T^0,T^\infty)$ of the $K$-class $[(T^0,T^\infty)]$ in $K_i(\Szi \bm{\otimes} C(X))$ is sufficiently stabilized to take $N = N'$ in Lemma~\ref{lemma2}.
Under this assumption, we take a continuous map $g$, a homotopy $\{ g_t \}_{0 \leq t \leq 1}$ from $f$ to $\pi^* g$ in $C(\T^2 \times X, GL_N^{(i)}(\C))$, a homotopy $\{ (T^0_t, T^\infty_t) \}_{0 \leq t \leq 1}$ and continuous maps $t^0$ and $t^\infty$ that satisfy the conditions in Lemma~\ref{lemma2}.
$t^0$ and $t^\infty$ provide elements $[t^0]$ and $[t^\infty]$ in $K^{-i}(\T \times X)$.
Let $r \colon \Z \times \Z_{\geq 0} \to \Z_{\geq 0} \times \Z$ be the map defined by $r(m,n)=(n,-m)$ and $r^* \colon l^2(\Z_{\geq 0} \times \Z, \C^{n_iN}) \to l^2(\Z \times \Z_{\geq 0}, \C^{n_iN})$ be the induced Hilbert space isomorphism.
For an operator $T \in M_{n_iN}(\TT^\infty)$, we write
\begin{equation}\label{tilde1}
	\widetilde{T} = r^* \circ T \circ (r^*)^{-1} \in M_{n_iN}(\TT^0).
\end{equation}
By taking this transformation for $T^\infty_1(w,x) = t^\infty(w,x) \oplus g(x)$ in Lemma~\ref{lemma2}, we obtain $\tilde{T}^\infty_1(z,x) = \tilde{t}^\infty(z,x) \oplus g(x)$ where $\tilde{t}^\infty \colon \T \times X \to GL^{(i)}_{LN'}(\C)$.
Then, we have $\partial^{\text{T}}([\tilde{t}^\infty]) = - \partial^{\text{T}}([t^\infty])$ in $K^{-i+1}(X)$ since the orientation of $\T$ is reversed through (\ref{tilde1}).
For these elements, the following holds \cite{CDS72}:

\begin{proposition}\label{CDSthm}
	$\partial^{\mathrm{qT}}([(T^0,T^\infty)]) = (\partial^{\mathrm{T}} \oplus - \partial^{\mathrm{T}})([t^0], [\tilde{t}^\infty])$.
\end{proposition}
\begin{proof}
We decompose the Hilbert space $l^2(\Z_{\geq 0} \times \Z_{\geq 0}, \C^{n_iN})$ into a direct sum of three closed subspaces,
\begin{equation*}
	l^2(\Z_{\geq 0} \times \Z_{\geq 0}, \C^{n_iN}) = \HH_1 \oplus \HH_2 \oplus \HH_3,
\end{equation*}
where
$\HH_1 = l^2(\Z_{\geq 0} \times \{ 0, \ldots, K-1 \}, \C^{n_iN})$,
$\HH_2 = l^2(\{ 0, \ldots, L-1 \} \times \Z_{\geq K}, \C^{n_iN})$ and
$\HH_3 = l^2(\Z_{\geq K} \times \Z_{\geq L}, \C^{n_iN})$.
We consider the following family of quarter-plane Toeplitz operators:
\begin{equation*}
A(x) =
\left\{
\begin{aligned}
T_{t^0(x)} & \hspace{3mm} \text{on} \ \HH_1 \cong l^2(\Z_{\geq 0}, \C^{n_iKN}),\\
T_{t^\infty(x)} & \hspace{3mm} \text{on} \ \HH_2 \cong l^2(\Z_{\geq 0}, \C^{n_iLN}),\\
g(x) & \hspace{3mm} \text{on} \ \HH_3,
\end{aligned}
\right.
\end{equation*}
where $x \in X$ (see Figure~\ref{Fig5}).
\begin{figure}
  \centering
  \includegraphics[width=5cm]{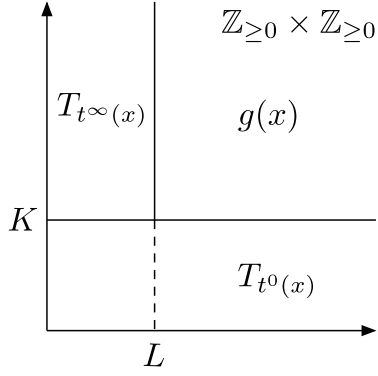}
  \caption{Schematic picture for the quarter-plane Toeplitz operator $A(x)$}
    \label{Fig5}
\end{figure}
$A$ is a family of (self-adjoint) Fredholm operators that satisfies $\gamma(A) = (T^0_1, T^\infty_1)$.
Since $g(x)$ on $\HH_3$ is (self-adjoint) invertible, the following equality holds in the $K$-group $K^{-i+1}(X)$.
\begin{gather*}
	(\partial^{\text{T}} \oplus - \partial^{\text{T}})([t^0], [\tilde{t}^\infty])
		= \partial^{\text{T}}([t^0]) + \partial^{\text{T}}([t^\infty])
		= [T_{t^0}] + [T_{t^\infty}]
		= [T_{t^0}] + [T_{t^\infty}] + [g]\\
		= [A]
		= \partial^{\text{qT}}([(T^0_1,T^\infty_1)])
		= \partial^{\text{qT}}([(T^0,T^\infty)]).
\end{gather*}
\end{proof}

Let $\{a_t\}_{0 \leq t \leq 1}$ (resp. $\{b_t\}_{0 \leq t \leq 1}$) be a path from $t^0$ (resp. $\tilde{t}^\infty$) to $\overbrace{g \oplus \cdots \oplus g}^{K \ (\text{resp.} \ L)}$ in $C(\T \times X, M^{(i)}_{n_iKN}(\C))$ (resp. $C(\T \times X, M^{(i)}_{n_iLN}(\C))$).
These paths exist since $M^{(i)}_{n}(\C)$ is contractible.
We now consider the family of Toeplitz operators $F \colon [-2,2] \times \T \times X \to M_{n_i N}(\TT)$
defined by
\begin{equation*}
F(s, z, x) =
\left\{
\begin{aligned}
T^0_{s+2}(z,x), &\hspace{3mm} -2 \leq s \leq -1,\\
a_{s+1}(z,x) \oplus g(x),
& \hspace{3mm} -1 \leq s \leq 0,\\
b_{1-s}(z,x) \oplus g(x),
& \hspace{3mm} 0 \leq s \leq 1,\\
\widetilde{T}^\infty_{2-s}(z,x), & \hspace{3mm} 1 \leq s \leq 2.
\end{aligned}
\right.
\end{equation*}
Note that $F(0,z,x) = g(x)$ is (self-adjoint) invertible and independent of $z \in \T$.
$F$ is a family of (self-adjoint) Fredholm operators on $[-2,2] \times \T \times X$, which is invertible on $([-2,-1] \sqcup [1,2]) \times \T \times X$,
therefore, it defines the following element of the relative $K$-group,
\begin{equation*}
	[F] \in K^{-i+1}([-2,2] \times \T \times X, ([-2,-1] \sqcup [1,2]) \times \T \times X).
\end{equation*}
We write
\begin{equation*}
	[[F]] \in K^{-i+1}([-2,2] \times \T \times X, ([-2,-1] \sqcup [1,2]) \times \T \times X) \bigl/ \Ker(\alpha)
\end{equation*}
for the class of $[F]$ in the quotient group, and we define
\begin{equation*}
	\psi \colon K_i(\Szi \bm{\otimes} C(X)) \to K^{-i+1}([-2,2] \times \T \times X, ([-2,-1] \sqcup [1,2]) \times \T \times X) \bigl/ \Ker(\alpha)
\end{equation*}
by $\psi[(T^0,T^\infty)] = [[F]]$.

\begin{proposition}
	$\psi$ is a well-defined group homomorphism.
\end{proposition}

\begin{proof}
We show that the element $[[F]]$ depends only on the $K$-class $[(T^0,T^\infty)]$ in $K_1(\Szi \bm{\otimes} C(X))$.
By the excision property of topological $K$-theory, we can remove components of $F$ corresponding to subspaces $[-2,-1)$ and $(1, 2]$ of $[-2,2]$.
That is, the inclusion $i \colon [-1,1] \hookrightarrow [-2,2]$ induces an isomorphism
\begin{gather*}
	(i \times \id_{\T \times X})^* \colon K^{-i+1}([-2,2] \times \T \times X, ([-2,-1] \sqcup [1,2]) \times \T \times X)\\
		\hspace{1.5cm} \overset{\cong}{\longrightarrow} K^{-i+1}([-1,1] \times \T \times X, \{\pm1\} \times \T \times X),
\end{gather*}
and satisfies $(i \times \id_{\T \times X})^*[F] = [F\lvert_{[-1,1] \times \T \times X}]$.
We identify these two $K$-groups by this map.
For simplicity of description, we take $K$ and $L$ in Lemma~\ref{lemma2} sufficiently large and assume $K=L$.
Let us consider the map $\tilde{F} \colon [-1,1] \times \T \times X \to M_{n_i K N}^{(i)}(\C)$ given by
\begin{equation*}
\tilde{F}(s, z, x) =
\left\{
\begin{aligned}
a_{s+1}(z,x),
& \hspace{3mm} -1 \leq s \leq 0,\\
b_{1-s}(z,x),
& \hspace{3mm} 0 \leq s \leq 1.
\end{aligned}
\right.
\end{equation*}
Then, $F\lvert_{[-1,1] \times \T \times X} = \tilde{F} \oplus g$, and $[F\lvert_{[-1,1] \times \T \times X}] = [\tilde{F} \oplus g] = [\tilde{F} \oplus I^{(i)}_N]$
in the group $K^{-i+1}(([-1,1], \{\pm1\}) \times \T \times X)$ by Kuiper's theorem.
Since $\tilde{F}(s=-1) = t^0$ and $\tilde{F}(s=1) = \tilde{t}^\infty$, $\tilde{F}$ is a (self-adjoint) lift of $t^0 \sqcup \tilde{t}^\infty \colon \{\pm1\} \times \T \times X \to GL_{KN}^{(i)}(\C)$ through the restriction map,
\begin{equation*}
	\mathrm{res} \colon M_{n_iKN}(C([-1,1] \times \T \times X)) \to M_{n_iKN}(C(\{ \pm 1\} \times \T \times X)).
\end{equation*}
Therefore, we have
$\partial^{\text{pair}}([t^0], [\tilde{t}^\infty]) = [F\lvert_{[-1,1] \times \T \times X}]$.
By the definition of the map $\bar{\alpha}$ and Proposition~\ref{CDSthm}, we have
\begin{equation*}
	\bar{\alpha}([[F]]) = (\partial^{\text{T}} \oplus - \partial^{\text{T}})([t^0], [\tilde{t}^\infty]) = \partial^{\text{qT}}([(T^0,T^\infty)]).
\end{equation*}
Since $\partial^{\text{qT}}([(T^0,T^\infty)])$ depends only on the $K$-class $[(T^0,T^\infty)]$ and $\bar{\alpha}$ is an isomorphism, the element $[[F]]$ is independent of the choice made to construct $F$, and depends only on the $K$-class $[(T^0,T^\infty)]$.
\end{proof}

\subsection{Construction of the Map $\phi$}
\label{Sect4.2}
In this subsection, we construct the following group homomorphism $\phi$ in the diagram (\ref{maindiagfamily}).

Let $f \colon \tilde{\SSS}^3 \times X \to GL^{(i)}_N(\C)$ be a continuous map that represents $[f] \in K^{-i}(\tilde{\SSS}^3 \times X)$, and let $f_r = f\lvert_{\T^2 \times X}$.
The inclusion $\T^2 \times X \hookrightarrow \tilde{\SSS}^3 \times X$ induces the homomorphism $K^{-i}(\tilde{\SSS}^3 \times X) \to K^{-i}(\T^2 \times X)$, and the image of this map is the same as that of $\pi^*$ in (\ref{piast}).
Therefore, as in the proof of Lemma~\ref{lemma2}, there exists $N' \geq N$, a continuous map $g \colon X \to GL^{(i)}_{N'}(\C)$ and a homotopy $\{ g_t \}_{0 \leq t \leq 1}$ from $g_0 = f_r \oplus I^{(i)}_{N'-N}$ to $g_1 = \pi^* g$ in $C(\T^2 \times X, GL^{(i)}_{N'}(\C))$.
As in Sect.~\ref{Sect4.1}, we assume that the representative $f$ of the $K$-class $[f]$ is sufficiently stabilized to take $N' = N$ for simplicity.

Let $R \colon \T^2 \to \T^2$ be an orientation-preserving homeomorphism given by $R(z,w) = (w, \bar{z})$.
Corresponding to the notation (\ref{tilde1}) for half-plane Toeplitz operators, we write $\tilde{k} = (R \times \id_X)^* k$ for a matrix-valued continuous function $k$ on $\T^2 \times X$; that is $\tilde{k}(z,w,x)=k(w,\bar{z},x)$.
By using this notation, let $G \colon [-2,2] \times \T^2 \times X \to GL^{(i)}_N(\C)$ be a continuous map defined as follows:
\begin{equation}\label{Tfamily}
G(s,z,w,x) =
\left\{
\begin{aligned}
f(z, (2+s)w,x), &\hspace{3mm} -2 \leq s \leq -1,\\
g_{1+s}(z,w,x), &\hspace{3mm} -1 \leq s \leq 0,\\
\tilde{g}_{1-s}(z,w,x) = g_{1-s}(w,\bar{z},x), & \hspace{3mm} 0 \leq s \leq 1,\\
\tilde{f}(z,(2-s)w,x) = f((2-s)w,\bar{z},x), & \hspace{3mm} 1 \leq s \leq 2.
\end{aligned}
\right.
\end{equation}
Since $g_1(z,w,x) = g(x)$ does not depend on $z$ and $w$, this family $G$ defines a continuous map on $[-2,2] \times \T^2 \times X$.
We consider a family of Toeplitz operators $T_G$ by using second $\T$ parametrized by $w$ for the symbols of Toeplitz operators; that is, for $(s,z,x) \in [-2,2] \times \T \times X$, we set
\begin{equation*}
	T_G(s,z,x) = T_{G(s,z,\cdot,x)}.
\end{equation*}
$T_G$ is a family of Toeplitz operators of invertible symbols and therefore a family of Fredholm operators parametrized by $[-2,2] \times \T \times X$.
When $s=\pm 2$, these Toeplitz operators have constant invertible symbols and are invertible.
This family $T_G$ provides the element $[T_G]$ of the $K$-group $K^{-i+1}(([-2,2],\{ \pm 2 \}) \times \T \times X)$.
We define 
\begin{equation*}
	\phi \colon K^{-i}(\tilde{\SSS}^3 \times X) \to K^{-i+1}([-2,2] \times \T \times X, \{ \pm 2\} \times \T \times X) \bigl/ \Ker(\alpha)
\end{equation*}
by $\phi([f]) = [[T_G]]$.

\begin{proposition}\label{welldefphi}
	$\phi$ is a well-defined group homomorphism.
\end{proposition}

\begin{proof}
We show that $[[T_G]]$ depends only on the $K$-class $[f]$ in $K^{-i}(\tilde{\SSS}^3 \times X)$.
The proof is divided into three steps.

(i) We first fix $f$ and $g$ and show that $[[T_G]]$ is independent of the choice of homotopy $g_t$.
For this purpose, we take another homotopy $\{h_t\}_{0 \leq t \leq 1}$ from $h_0 = f_r$ to $h_1 = \pi^* g$ in $C(\T^2 \times X, GL_N^{(i)}(\C))$.
As in (\ref{Tfamily}), we define
$H \colon [-2,2] \times \T^2 \times X \to GL^{(i)}_N(\C)$
by using $h_t$ in place of $g_t$, and we consider a family of Toeplitz operators $T_H$ defined as in $T_G$.
In what follows, we show that
\begin{equation}\label{phiproof}
	[[T_G]] = [[T_H]] \in K^{-i+1}([-2,2] \times \T \times X, \{ \pm 2 \} \times \T \times X) \bigl/ \Ker(\alpha).
\end{equation}
Let $q \colon [-2,2] \to [-2,2]$ be the map given by $q(s) = -s$, and let $H' = (q \times \id_{\T^2 \times X})^*(H)$.
Note that $(q \times \id_{\T \times X})^*(T_H) = T_{H'}$ and the isomorphism
$(q \times \id_{\T \times X})^*$ on $K^{-i+1}(([-2,2],\{ \pm 2 \}) \times \T \times X)$
is a multiplication by $-1$.
Therefore, 
	$[T_{H'}] = -[T_H]$ in $K^{-i+1}(([-2,2],\{ \pm 2 \}) \times \T \times X)$.
Let us consider the following map,
\begin{gather*}
	K^{-i+1}(([-2,2], \{ \pm 2 \})  \times \T\times X) \times K^{-i+1}(([-2,2], \{ \pm 2 \})  \times \T\times X)\\
	\hspace{5cm} \overset{\mathcal{G}}{\longrightarrow} K^{-i+1}(\T^2 \times X),
\end{gather*}
which is constructed as follows:
\begin{itemize}
\item For spaces, two endpoints of two intervals $[-2,2]$ are connected (connect $\pm 2$ with $\mp 2$) to construct the circle $\T$.
\item For two families of Fredholm operators that are invertible on $\{\pm 2\} \times \T \times X$, we connect invertible operators on two endpoints correspondingly by a continuous path of (self-adjoint) invertible families. This is possible and unique up to homotopy by Kuiper's theorem.
\item We then obtain a family of Fredholm operators on $\T^2 \times X$, which defines an element of $K^{-i+1}(\T^2 \times X)$.
\end{itemize}
Note that we have the following isomorphism:
\begin{align}
\label{Kgroup}
\begin{aligned}
	K^{-i+1}(\T^2\times X)
		&\cong K^{-i+1}(\T \times X) \oplus K^{-i+1}([-2,2] \times \T\times X, \{ \pm 2 \} \times \T \times X)\\
		&\cong K^{-i+1}(X) \oplus K^{-i}(X)^{\oplus 2} \oplus K^{-i-1}(X).
\end{aligned}
\end{align}
Let $\iota \colon K^{-i+1}([-2,2] \times \T\times X, \{ \pm 2 \} \times \T \times X) \hookrightarrow K^{-i+1}(\T^2\times X)$ be the inclusion corresponding to the direct sum decomposition in (\ref{Kgroup}).
The map $\mathcal{G}$ is the same as the composite of the addition of the group $K^{-i+1}(([-2,2],\{ \pm 2 \}) \times \T \times X)$ followed by $\iota$.
Let us consider the projection $\varpi \colon K^{-i+1}(\T^2\times X) \to K^{-i-1}(X)$
corresponding to the decomposition in (\ref{Kgroup}).
Through the identification (\ref{identifyK}), $[[T_G]]$ and $\varpi \circ \iota([T_G])$ corresponds.
Therefore, in order to show (\ref{phiproof}), it is sufficient to show that
\begin{equation}\label{targeteqn}
	\varpi \circ \iota([T_G]) = \varpi \circ \iota([T_H]) \in K^{-i-1}(X).
\end{equation}

Let us discuss the element $\mathcal{G}([T_G], [T_{H'}])$.
The families of Toeplitz operators $T_G$ and $T_{H'}$ have symbols $G$ and $H'$, respectively, which are shown in Figure~\ref{Fig2}.
\begin{figure}
  \centering
  \includegraphics[width=12.5cm]{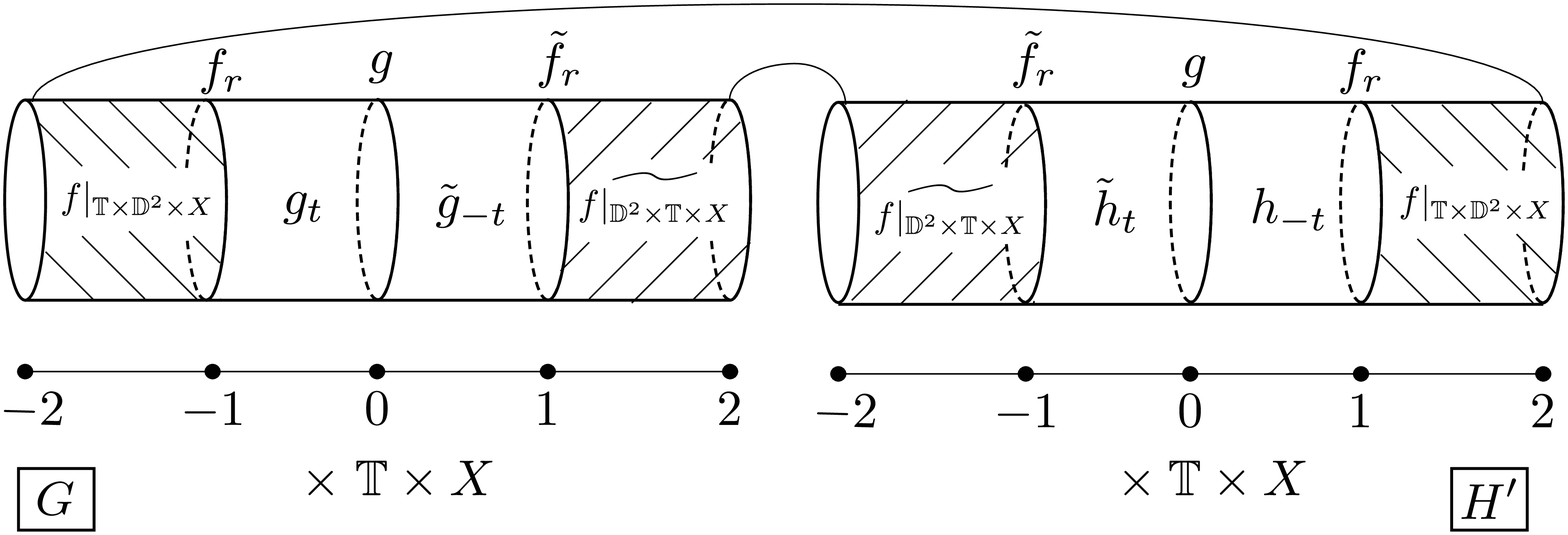}
  \caption{Schematic picture for matrix functions $G$ and $H'$. The signs of the subscripts of homotopies (e.g. $g_t$ or $h_{-t}$) indicate that the orientation of the interval $[-2,2]$ is consistent with that for each homotopy (e.g., $g_t$) or not (e.g., $h_{-t}$)}
    \label{Fig2}
\end{figure}
Since these two have the same matrix functions at the boundaries, we can glue $G$ and $H'$ as indicated in Figure~\ref{Fig2}, and obtain a continuous map $\T^3 \times X \to GL^{(i)}_N(\C)$ that defines an element of the group $K^{-i}(\T^3 \times X)$.
Let $\partial^{\text{T}} \colon K^{-i}(\T^3 \times X) \to K^{-i+1}(\T^2 \times X)$ be the push-forward map through Toeplitz operators
with respect to the last variable of $\T^3$.
Then, this $K$-group element maps to $\mathcal{G}([T_G], [T_{H'}])$ through $\partial^{\text{T}}$.
On the two pairs of dashed areas in Figure~\ref{Fig2}, there are the same matrix functions with opposite parameter directions; therefore these dashed areas can be cancelled by a homotopy, and this $K$-class is the same as the $K$-class of the continuous map $\mathfrak{a} \colon [-2,2]/\{\pm 2\} \times \T^2 \times X \to GL^{(i)}_N(\C)$ given by,
\begin{equation*}
\mathfrak{a}(s,z,w,x) =
\left\{
\begin{aligned}
g_{2+s}(z,w,x), &\hspace{3mm} -2 \leq s \leq -1,\\
\tilde{g}_{-s}(z,w,x) = g_{-s}(w,\bar{z},x), &\hspace{3mm} -1 \leq s \leq 0,\\
\tilde{h}_s(z,w,x) = h_s(w,\bar{z},x), & \hspace{3mm} 0 \leq s \leq 1,\\
h_{2-s}(z,w,x), & \hspace{3mm} 1 \leq s \leq 2,
\end{aligned}
\right.
\end{equation*}
as indicated in Figure~\ref{Fig3}.
\begin{figure}
  \centering
  \includegraphics[width=8cm]{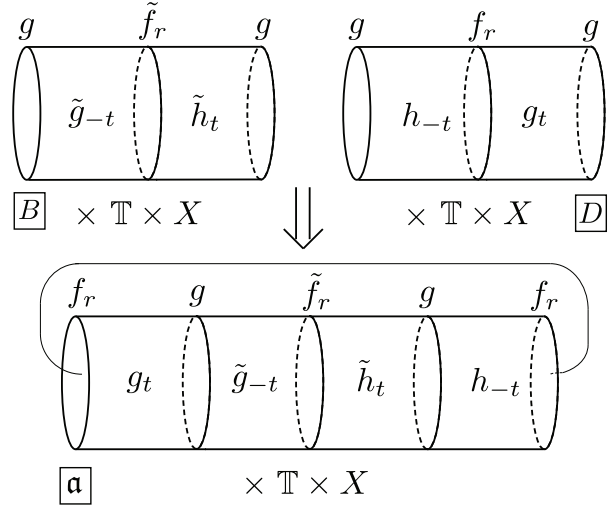}
  \caption{Schematic picture for $\mathfrak{a}$, $B$ and $D$}
    \label{Fig3}
\end{figure}
In summary, we have
\begin{equation*}
	\partial^{\text{T}}([\mathfrak{a}]) = [T_\mathfrak{a}] = \mathcal{G}([T_G],[T_{H'}]) = \iota([T_G]) + \iota([T_{H'}]) = \iota([T_G]) - \iota([T_H]).
\end{equation*}
Let $B, D \colon [-1,1] \times \T^2 \times X \to GL^{(i)}_N(\C)$ be given as $B = \mathfrak{a}\lvert_{-1 \leq s \leq 1}$ and
\begin{equation*}
D(s,z,w,x) =
\left\{
\begin{aligned}
h_{-s}(z,w,x), & \hspace{3mm} -1 \leq s \leq 0,\\
g_s(z,w,x) = g_{-s}(z,w,x), &\hspace{3mm} 0 \leq s \leq 1.
\end{aligned}
\right.
\end{equation*}
As shown in Figure~\ref{Fig3}, $B$ and $D$ correspond to two pieces obtained by cut $\mathfrak{a}$ along $s = \pm 1$.
Both  $B$ and $D$ have the same value at boundaries $s = \pm 1$ and reduce to the map on $[-1,1]/\{ \pm 1\} \times \T^2 \times X$, for which we write $\mathfrak{b}$ and $\mathfrak{d}$, respectively.
\begin{equation*}
	\mathfrak{b}, \mathfrak{d} \colon [-1,1]/\{ \pm 1\} \times \T \times \T \times X \to GL^{(i)}_N(\C).
\end{equation*}
For $B$, $D$, $\mathfrak{b}$ and $\mathfrak{d}$, we associate families of Toeplitz operators $T_B$, $T_D$, $T_\mathfrak{b}$ and $T_\mathfrak{d}$ by taking right $\T$ for the symbols.
These are families of (self-adjoint) Fredholm operators, and since $B(s = \pm 1)$ and $D(s = \pm 1)$ are given by (the pull-back of) $g$, both $T_B(s=\pm1)$ and $T_D(s=\pm1)$ are invertible.
Therefore, they define the following elements of $K$-groups:
\begin{equation*}
	[T_B], [T_D] \in K^{-i+1}([-1,1] \times \T \times X, \{ \pm 1 \} \times \T \times X),
\end{equation*}
\begin{equation*}
	[T_\mathfrak{b}], [T_\mathfrak{d}] \in K^{-i+1}([-1,1]/\{ \pm 1\} \times \T \times X).
\end{equation*}
Since $\mathfrak{a}$ is obtained by gluing $B$ and $D$, we have
\begin{equation*}
	[T_\mathfrak{a}] = \mathcal{G}([T_B], [T_D]) = \iota([T_B]) + \iota([T_D]) = [T_\mathfrak{b}] + [T_\mathfrak{d}].
\end{equation*}
Let us consider the homeomorphism $q \times R$ on $[-1,1]/\{ \pm 1 \} \times \T^2$ given by $(q \times R)(s, z, w) = (-s, w, \bar{z})$,
and consider the map $(q \times R \times \id_X)^*$ on the $K$-group $K^{-i}(\T^3 \times X)$.
Then $(q \times R \times \id_X)^*[\mathfrak{d}] = [\mathfrak{b}]$ holds.
Furthermore, corresponding to the direct sum decomposition
\begin{equation}\label{Kgroup2}
	K^{-i}(\T^3 \times X) \cong
		K^{-i}(X) \oplus K^{-i-1}(X)^{\oplus 3} \oplus K^{-i-2}(X)^{\oplus 3} \oplus K^{-i-3}(X),
\end{equation}
the map $(q \times R \times \id_X)^*$ acts on the direct summand $K^{-i-3}(X)$ as multiplication by $-1$.
By the map $\partial^{\text{T}}$, this direct summand $K^{-i-3}(X)$ in (\ref{Kgroup2}) maps isomorphically to the direct summand $K^{-i-1}(X)$ in (\ref{Kgroup}).
Therefore, we have $\varpi ([T_\mathfrak{b}]) = - \varpi ([T_\mathfrak{d}])$,
which leads to equation (\ref{targeteqn}); that is,
\begin{equation*}
	\varpi \circ \iota([T_G]) - \varpi \circ \iota([T_H]) 
	= \varpi \circ \mathcal{G}([T_G], [T_{H'}]) 
	= \varpi ([T_\mathfrak{a}])  = \varpi ([T_\mathfrak{b}]) + \varpi ([T_\mathfrak{d}]) = 0.
\end{equation*}

(ii) We next fix $f$ and show that $[[T_G]]$ is independent of the choice of a representative $g$ of the $K$-class $[g]$ in $K^{-i}(X)$ satisfying $[f_r] = \pi^*[g]$.
Note that the $K$-class $[g]$ is uniquely determined since $\pi^*$ is injective.
Let $g' \colon X \to GL^{(i)}_{N'}(\C)$ be another representative; that is, $[g]=[g']$.
By taking $N$ and $N'$ sufficiently large, we assume that $N = N'$ and that there exists a homotopy $\{k_t\}_{0 \leq t \leq 1}$ from $k_0 = g$ to $k_1 = g'$ in $C(X, GL^{(i)}_N(\C))$.
Combined with the homotopies $\{g_t\}_{0 \leq t \leq 1}$ (from $f_r$ to $\pi^* g$) and $\{\pi^* k_t\}_{0 \leq t \leq 1}$, we obtain a homotopy from $f_r$ to $\pi^* g'$ in $C(\T^2 \times X, GL^{(i)}_N(\C))$.
The associated family of Toeplitz operators is expressed as $T_K$, where
$K \colon [-2,2] \times \T^2 \times X \to GL^{(i)}_N(\C)$ is a continuous map given by
\begin{equation*}
K(s,z,w,x) =
\left\{
\begin{aligned}
f(z, (2+s)w,x), &\hspace{3mm} -2 \leq s \leq -1,\\
g_{2+2s}(z,w,x), &\hspace{3mm} -1 \leq s \leq -\frac{1}{2},\\
\pi^*k_{1+2s}(z,w,x) = k_{1+2s}(x), &\hspace{3mm} -\frac{1}{2} \leq s \leq 0,\\
\widetilde{\pi^* k}_{1-2s}(z,w,x) = k_{1-2s}(x), & \hspace{3mm} 0 \leq s \leq \frac{1}{2},\\
\tilde{g}_{2-2s}(z,w,x), & \hspace{3mm} \frac{1}{2} \leq s \leq 1,\\
\tilde{f}(z,(2-s)w,x), & \hspace{3mm} 1 \leq s \leq 2.
\end{aligned}
\right.
\end{equation*}
Within $-\frac{1}{2} \leq s \leq \frac{1}{2}$, this family $K$ has $k_t$ of the opposite parameter direction, and it can be collapsed by a homotopy to obtain $G$.
This homotopy preserves $s = \pm 2$ and provides a homotopy from $T_G$ to $T_K$ in the space of families of (self-adjoint) Fredholm operators that are invertible at $s = \pm 2$; therefore, $[T_G] = [T_K]$.

(iii) Finally, we show that $[[T_G]]$ is independent of the choice of a representative $f$ of the $K$-class $[f]$ in $K^{-i}(\tilde{\SSS}^3 \times X)$.
Let $f' \colon \tilde{\SSS}^3 \times X \to GL^{(i)}_{N'}(\C)$ be another representative; that is, $[f'] = [f]$.
As in (2), we assume $N = N'$ and that there exists a homotopy $\{f_t\}_{0 \leq t \leq 1}$ from $f_0 = f$ to $f_1 = f'$ in $C(\tilde{\SSS}^3 \times X, GL^{(i)}_N(\C))$.
In this case, $\{(f_t)_r\}_{0 \leq t \leq 1}$ provides a path from $f_r$ to $f'_r$ in $C(\T^2 \times X, GL^{(i)}_N(\C))$.
Combined with the homotopy $\{g_t\}_{0 \leq t \leq 1}$, we obtain a homotopy from $f_r$ to $\pi^* g$, which provides the family of Toeplitz operators $T_{G'}$ and a homotopy from $T_G$ to $T_{G'}$.
Explicitly, for $0 \leq u \leq 1$, let $G_u \colon [-2,2] \times \T^2 \times X \to GL^{(i)}_N(\C)$
be a continuous map defined as follows:
\begin{equation*}
G_u(s,z,w,x) =
\left\{
\begin{aligned}
f_u(z, (2+s)w,x), &\hspace{3mm} -2 \leq s \leq -1,\\
f_{-2s+u-2}(z,w,x), &\hspace{3mm} -1 \leq s \leq -1 + \frac{u}{2},\\
g_{\frac{2}{2-u}s + 1}(z,w,x), &\hspace{3mm} -1+ \frac{u}{2} \leq s \leq 0,\\
\tilde{g}_{\frac{-2}{2-u}s + 1}(z,w,x), & \hspace{3mm} 0 \leq s \leq 1-\frac{u}{2},\\
\tilde{f}_{2s+u-2}(z,w,x), & \hspace{3mm} 1-\frac{u}{2} \leq s \leq 1,\\
\tilde{f}(z,(2-s)w,x), & \hspace{3mm} 1 \leq s \leq 2.
\end{aligned}
\right.
\end{equation*}
Then, the associated family of Toeplitz operators $\{T_{G_u}\}_{u \in [0,1]}$ provides a path of families of (self-adjoint) Fredholm operators invertible at $s = \pm 2$ satisfying $T_{G_0} = T_G$ and $T_{G_1} = T_{G'}$; therefore, $[T_G] = [T_{G'}]$.
\end{proof}
The following lemma states that the diagram (\ref{maindiagfamily}) commutes:
\begin{lemma}
	$\bar{\alpha} \circ \phi = \beta$.
\end{lemma}
\begin{proof}
The map $\beta$ is given through the decomposition (\ref{Kdecomp}) and the square of the Bott periodicity isomorphism.
The map $\bar{\alpha}$ corresponds to the Bott periodicity isomorphism, and we see a relation between the map $\phi$ and the Bott periodicity.

Let $\tpi \colon \tilde{\SSS}^3 \times X \to X$ be the projection.
Through the decomposition (\ref{Kdecomp}), the direct summand $K^{-i}(X)$ in $K^{-i}(\tilde{\SSS}^3 \times X)$ is represented by $[\tpi^* g]$, where $g \colon X \to GL^{(i)}_N(\C)$ for some $N$.
$\tpi^* g$ is a continuous map on $\tilde{\SSS}^3 \times X$, which is constant with respect to $\tilde{\SSS}^3$.
In this case, a homotopy $\{ g_t \}_{0 \leq t \leq 1}$ to construct $G$ in (\ref{Tfamily}) can be taken to be constant $g_t = \pi^* g$, and
the corresponding family $T_G$ of Toeplitz operators representing $\phi([\tpi^* g])$ to be a family of (self-adjoint) invertible operators.
Therefore, $\phi([\tpi^* g]) = 0$, and the direct summand $K^{-i}(X)$ maps to zero through $\phi$.

We next consider the direct summand of $K^{-i}(\tilde{\SSS}^3 \times X)$ corresponding to $K^{-i}_\cpt(\R^3 \times X)$ in (\ref{Kdecomp}).
We consider $\tilde{\SSS}^3$ as the one-point compactification of $\R^3$, where the point at infinity corresponds to the base point $s_0 = (1,1) \in \tilde{\SSS}^3$.
$K$-classes in this component are represented by a continuous maps $f \colon \tilde{\SSS}^3 \times X \to GL^{(i)}_N(\C)$ satisfying $f(s_0, x) = I^{(i)}_N$ for any $x \in X$.
Since $\tilde{\SSS}^3 \setminus \{ s_0 \} \cong \R^3$ is contractible, we can take a small three-dimensional ball $\mathbb{B}$ in $\tilde{\SSS}^3$ and deform $f$ to a continuous map which is $I^{(i)}_N$ on $(\tilde{\SSS}^3 \setminus \mathbb{B}) \times X$.
Therefore, the $K$-class $[f]$ is represented by a continuous map $f' \colon \tilde{\SSS}^3 \times X \to GL^{(i)}_N(\C)$, which is $I^{(i)}_N$ on $\D^2 \times \T$ and $\T \times \{ 0 \}$, where $0$ is the origin of $\D^2 \subset \C$.
Since $f'_r = I^{(i)}_N$, we can take a homotopy $\{ g_t \}_{0 \leq t \leq 1}$ to construct $G$ in (\ref{Tfamily}) as $g_t = I^{(i)}_N$.
Then, the map $G \colon [-2,2] \times \T^2 \times X \to GL^{(i)}_N(\C)$ is $I^{(i)}_N$ when $s=-2$ and $-1 \leq s \leq 2$, and corresponding family of Toeplitz operators $T_G$ is also $I^{(i)}_N$ on $s=-2$ and $-1 \leq s \leq 2$.
Therefore, $\phi$ on $K^{-i}_\cpt(\R^3 \times X)$ is simply the push-forward map through the Toeplitz operators \cite{At68}.
Explicitly, it is the same as the composite of the following maps:
\begin{gather*}
	K^{-i-3}(X) \cong K^{-i}_\cpt(\R^3 \times X) \to K^{-i}_\cpt(\R \times \T^2 \times X)
		\overset{\partial^{\text{T}}}{\longrightarrow} K^{-i+1}_\cpt(\R \times \T \times X)\\
		 \overset{/\Ker(\alpha)}{\longrightarrow} K^{-i+1}([-2,2] \times \T \times X,\{ \pm 2\} \times \T \times X) / \Ker(\alpha) \cong K^{-i-1}(X).
\end{gather*}
In the above, the map from $K^{-i}_\cpt(\R^3 \times X)$ to $K^{-i}_\cpt(\R \times \T^2 \times X)$ is constructed as follows: we embed $\R^3$ into $\R \times \T^2$ as a small open ball.
A $K$-class in $K^{-i}_\cpt(\R^3 \times X)$ is represented by a continuous map $g$ from $\R^3 \times X$ to $GL^{(i)}_N(\C)$ for some $N$, which is $I^{(i)}_N$ outside of a compact set.
The image of the $K$-class $[g]$ in $K^{-i}_\cpt(\R \times \T^2 \times X)$ is represented by the extension of $g$ to $\R \times \T^2 \times X$ by $I^{(i)}_N$.
Therefore, $\phi$ on $K^{-i}_\cpt(\R^3 \times X)$ is given by the Bott periodicity isomorphism \cite{At68}.

In summary, the composite $\bar{\alpha} \circ \phi$ is a projection onto $K^{-i}_\cpt(\R^3 \times X)$ followed by the square of the Bott periodicity isomorphism, which is $\beta$.
\end{proof}

\subsection{Proof of Theorem~\ref{mainfamily}}
\label{Sect.4.3}
Theorem~\ref{mainfamily} follows from the proposition below and the commutativity of the diagram (\ref{maindiagfamily}).
\begin{proposition}\label{mainpsingle}
Let $f \colon \T^2 \times X \to GL^{(i)}_N(\C)$ be a continuous map such that for each $x \in X$, $f(x)$ is a two-variable rational matrix function such that the associated quarter-plane Toeplitz operator $T^{0,\infty}_{f(x)}$ is Fredholm.
Let $f^E \colon \tilde{\SSS}^3 \times X \to GL^{(i)}_N(\C)$ be the extension of $f$ through matrix factorization in the hermitian case of Proposition~\ref{exthf} (when $i=0$) or Proposition~\ref{ext3} (when $i=1$).
We then have $\psi([(T^0_f,T^\infty_f)]) = \phi([f^E])$.
\end{proposition}

\begin{proof}
For the definition of $\psi([(T^0_f,T^\infty_f)])$, we take a path $\{ (T^0_t, T^\infty_t) \}_{0\leq t \leq 1}$ in $GL^{(i)}_{N'}(\Szi \bm{\otimes} C(X))$ in Lemma~\ref{lemma2} and construct a family $F$ of (self-adjoint) Fredholm operators on $[-2,2] \times \T^2 \times X$ representing the element $\psi([(T^0_f,T^\infty_f)])$ ((b) in Figure~\ref{Fig4}).
Note that $T^0_f$ is identified with $T_f$ through the isomorphism $\TT^0 \cong C(\T) \bm{\otimes} \TT$ and that $\tilde{T}^\infty_f$ is identified with $T_{\tilde{f}}$.
Associated with the family $\{ (T^0_t, T^\infty_t) \}_{0\leq t \leq 1}$, there is a path\footnote{We assume that $f$ is sufficiently stabilized to take $N' = N$ for simplicity.} $g_t =  \sigmaz(T^0_t) = \sigmai(T^\infty_t)$ from $f$ to $\pi^* g$ in $C(\T^2 \times X, GL^{(i)}_{N}(\C))$.
We use this path $\{ g_t \}_{0 \leq t \leq 1}$ to define $\phi([f^E])$.
Families of (self-adjoint) Fredholm Toeplitz operators $T_G$ representing the $K$-group element $\phi([f^E])$ and their symbols $G$ are indicated in (d) and (c) of Figure~\ref{Fig4}.
\begin{figure}
  \centering
  \includegraphics[width=12.5cm]{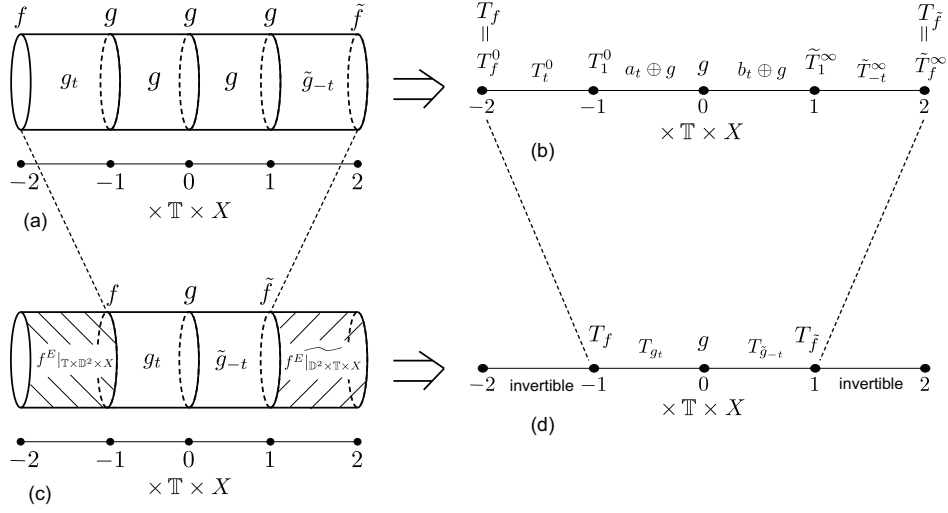}
  \caption{Families of Toeplitz operators and their symbols used to define $\psi([(T^0_f,T^\infty_f)])$ ((b) and (a)) and $\phi([f^E])$ ((d) and (c)): (a) symbols of $F$, (b) family of Toeplitz operators $F$, (c) symbols $G$ of $T_G$, (d) family of Toeplitz operators $T_G$}
    \label{Fig4}
\end{figure}
Note that $f^E_r = f^E\lvert_{\T^2 \times X} = f$ in our setup.
Since $f^E$ in Propositions~\ref{ext3} and \ref{exthf} is the extension of $f$ constructed through matrix factorization, the families of Toeplitz operators $T_G$ are (self-adjoint) invertible on $([-2,-1] \sqcup [1,2]) \times \T \times X$ by Lemma~\ref{extension3}.

Let us compare $F$ and $T_G$.
Let $p \colon [-2,2] \to [-1,1]$ be a map defined as follows:
\begin{equation*}
p(s) =
\left\{
\begin{aligned}
s+1, &\hspace{3mm} -2 \leq s \leq -1,\\
0, &\hspace{3mm} -1 \leq s \leq 1,\\
s-1, & \hspace{3mm} 1 \leq s \leq 2,
\end{aligned}
\right.
\end{equation*}
that is, collapsing the subinterval $[-1,1]$ in $[-2,2]$ to zero.
We take a restriction of $T_G$ onto $-1 \leq s \leq 1$ and extend $\{s=0\}$ to $-1 \leq s \leq 1$; that is, we consider
\begin{equation*}
	T'_G = (p \times \id_{\T \times X})^* (T_G\lvert_{[-1,1] \times \T \times X}).
\end{equation*}
Since $T_G$ is invertible when $-2 \leq s \leq -1$ and $1 \leq s \leq 2$, $T'_G$ is also invertible when $s = \pm 2$, and the $K$-class of $T'_G$ in $K^{-i+1}([-2,2] \times \T \times X, \{ \pm 2 \}  \times \T \times X)$ is the same as that of $T_G$.
Families $F$ and $T'_G$ coincide at boundaries $s = \pm 2$, and on the interior, these two are families of (self-adjoint) Fredholm Toeplitz operators of the same symbol.
Therefore, they differ by (self-adjoint) compact operators (see Figure~\ref{Fig4}), and the linear path $(1-t)F + t T'_G$ for $0 \leq t \leq 1$ provides a homotopy between $F$ and $T'_G$ in the space of (self-adjoint) Fredholm operators on $[-2,2] \times \T \times X$, which is invertible on $\{ \pm 2 \} \times \T \times X$.
We thus have,
\begin{equation*}
	[F] = [T'_G] = [T_G] \in K^{-i+1}([-2,2] \times \T \times X, \{ \pm 2 \}  \times \T \times X).
\end{equation*}
By taking a quotient by $\Ker(\alpha)$, we obtain $\psi([(T^0_f,T^\infty_f)]) = \phi([f^E])$.
\end{proof}
The following is an example of Theorem~\ref{mainfamily} when $i=1$ and $X = \{ \pt \}$.

\begin{example}\label{example}
Let $f \colon \T^2 \to GL(2,\C)$ be the two-variable rational matrix function given by
$f(z,w) =
	\begin{pmatrix}
z & -w^{-1}\\
w & z^{-1}
\end{pmatrix}$.
We first consider $f$ as a family of rational matrix functions on the circle $\T$ of variable $z$ parametrized by $w \in \T$.
By applying an algorithm from \cite{GKS03}, we obtain a right factorization of the following form:
\begin{equation*}
		f(z,w) = f_-(z,w) \cdot f_+(z,w) =
\begin{pmatrix}
1 & 0\\
-wz^{-1} & 1
\end{pmatrix}
\begin{pmatrix}
z & -w^{-1}\\
2w & 0
\end{pmatrix}.
\end{equation*}
Since this is a canonical factorization, the half-plane Toeplitz operator $T^\infty_f$ is invertible.
Our extended symbol $f^E$ in Proposition~\ref{ext3} is given as,
\begin{equation}\label{exeq1}
		f^E(z,w) = f_-(\bar{z}^{-1},w) \cdot f_+(z,w) =
\begin{pmatrix}
z & -w^{-1}\\
-w\lvert z\lvert^2 + 2w & \bar{z}
\end{pmatrix},
\end{equation}
for $(z,w) \in \D^2 \times \T$.
We next consider $z \in \T$ as a parameter, and by a similar discussion for rational matrix functions $f$ with respect to $w \in \T$, we also obtain a canonical factorization. Therefore, the half-plane Toeplitz operator $T^0_f$ is invertible and $f^E$ for $(z,w) \in \T \times \D^2$ is given as
\begin{equation}\label{exeq2}
		f^E(z,w) = 
\begin{pmatrix}
-z\lvert w \lvert^2 + 2z & -\bar{w}\\
w & z^{-1}
\end{pmatrix}.
\end{equation}
For $0 \leq t \leq 1$, let $g_t \colon \tilde{\SSS}^3 \to GL(2,\C)$ be a continuous map given as follows:
On $\D^2 \times \T$, we set $g_t$ as the replacement of the $(2,1)$ component of $f^E$ in (\ref{exeq1}) with $tw(-\lvert z \lvert^2+1) + w$.
On $\T \times \D^2$ we set $g_t$ as the replacement of the $(1,1)$ component of $f^E$ in (\ref{exeq2}) with $tz(-\lvert w \lvert^2+1) + z$.
$g_t$ provides a homotopy between $g_1 = f^E$ and
$g_0(z,w) = 
\begin{pmatrix}
z & -\bar{w}\\
w & \bar{z}
\end{pmatrix}$
in $C(\tilde{\SSS}^3, GL(2,\C))$.
Therefore, $[f^E] = [g_0]$ in $K^{-1}(\tilde{\SSS}^3)$.
Note that the map $\beta \colon K^{-1}(\tilde{\SSS}^3) \to K^0( \{ \pt \}) \cong \Z$ is the square of the Bott periodicity isomorphism.
Since $g_0$ is the square of a Bott element, which is a generator of the $K$-group $K^{-1}(\tilde{\SSS}^3) \cong \Z$, we have $\beta([f^E]) = \beta([g_0]) = 1$.
Since $T^0_f$ and $T^\infty_f$ are invertible, the quarter-plane Toeplitz operator $T^{0,\infty}_f$ is Fredholm, and its Fredholm index is $1$ by Theorem~\ref{mainfamily}, which is also computed in \cite{CDS72, Dud77} by different methods.
\end{example}

Let us consider the special case of Theorem~\ref{mainfamily} when $i=1$ and $X = \{ \pt \}$.
Our extended symbol defines an element $[f^E] \in K^{-1}(\tilde{\SSS}^3)$ and, through the isomorphism $\beta \colon K^{-1}(\tilde{\SSS}^3) \overset{\cong}{\longrightarrow} \Z$, we obtain an integer $W_3(f^E) := \beta([f^E])$.
This integer corresponds to $\pi_3(GL_n(\C)) \cong \Z$ for $n \geq 2$ and is the three-dimensional winding number.
Therefore, we obtain the following result.

\begin{corollary}\label{maincorollary}
Let $f \colon \T^2 \to GL_N(\C)$ be a two-variable rational matrix function such that the associated quarter-plane Toeplitz operator $T^{0,\infty}_{f}$ is Fredholm.
Let $f^E \colon \tilde{\SSS}^3 \to GL_N(\C)$ be the extension of $f$ through matrix factorization in Proposition~\ref{ext3}.
Then the Fredholm index of $T^{0,\infty}_f$ is given by,
\begin{equation*}
	 \ind (T^{0,\infty}_f) = W_3(f^E).
\end{equation*}
\end{corollary}

\section{Quarter-Plane Toeplitz Operators Preserving Real Structures}
\label{Sect.5}
In this section, we discuss a variant of Theorem~\ref{mainfamily} in real $K$-theory.
For this purpose, we use Atiyah's $KR$-theory for spaces equipped with involutions \cite{At66} and Boersema--Loring's formulation for the $KO$-theory of real \CA s \cite{BL16}.
Index theory for quarter-plane Toeplitz operators preserving some real structures and application to topological corner states are discussed in \cite{Hayashi4}, which we mainly follow.

\subsection{Matrix Functions Preserving Real or Quaternionic Structures}
\label{Sect.5.1}
Let $\fC$ be a real or a quaternionic structure on $\C^n$, that is, an antiunitary operator on $\C^n$ whose square is $+1$ or $-1$.
Note that, when we consider a quaternionic structure, the positive integer $n$ must be even.
We write $\Ad_{\fC}$ for a real linear automorphism of order two on $M_n(\C)$ given by $\Ad_{\fC}(x) = \fC x \fC^*$.
We also write $*$ for the operation on $M_n(\C)$ taking the hermitian conjugate of matrices.
We write $c$ for complex conjugation on $\C$, that is, $c(z) = \bar{z}$.
Then, $(\T, c)$ is a $\Z_2$-space, which is a $\Z_2$-subspace of $(\D^2, c)$.

\begin{lemma}\label{frlemma1}
Let $\mathfrak{I}$ be a $\Z_2$-space {\upshape (i)} $(GL_n(\C), \Ad_\fC)$ or {\upshape (ii)} $(GL_n(\C), \Ad_\fC \circ *)$.
Let $(X, \zeta)$ be a $\Z_2$-space and
$f \colon (\T \times X, c \times \zeta) \to \mathfrak{I}$
be a $\Z_2$-equivariant continuous map such that for each $x \in X$, $f(x)$ is a rational matrix function of trivial partial indices.
Then, its extension $f^e$ in Lemma~\ref{extension2} is a $\Z_2$-map
$f^e \colon (\D^2 \times X, c \times \zeta) \to \mathfrak{I}$.
\end{lemma}

\begin{proof}
By Lemma~\ref{extension2}, the map $f^e$ is continuous.
We now show its $\Z_2$-equivalence.
For each $x \in X$, we take a canonical factorization $f(x) = f_-(x) f_+(x)$.

(i) We first consider the case in which $\mathfrak{I} = (GL_n(\C), \Ad_\fC)$.
By assumption, the following equation holds for $(z,x) \in \T \times X$:
\begin{equation}\label{factnreal1}
	f_-(z,x) f_+(z,x) = f(z,x) = \fC f(\bar{z},\zeta(x)) \fC^* = \fC f_-(\bar{z}, \zeta(x))\fC^* \cdot \fC f_+(\bar{z}, \zeta(x)) \fC^*.
\end{equation}
Let $f_+^e(x)$ and $f_-^e(x)$ be continuous extensions of $f_+(x)$ and $f_-(x)$ onto $\T \cup D_+$ and $\T \cup D_-$, which are analytic on $D_+$ and $D_-$, respectively.
The function $z \mapsto \fC f_+(\bar{z}, \zeta(x)) \fC^*$ is a rational matrix function on $\T$ and $\fC f^e_+(\bar{z}, \zeta(x)) \fC^*$ provides its continuous extension onto $\T \cup D_+$ which is analytic on $D_+$ as a nonsingular matrix function.
A similar observation holds for $\fC f_-(\bar{z}, \zeta(x))\fC^*$ and
the equation (\ref{factnreal1}) provides two canonical factorizations of $f(x)$.
Therefore, for each $x \in X$, there exists $B \in GL_n(\C)$ such that $B f_+(z, x) = \fC f_+(\bar{z}, \zeta(x)) \fC^*$ and $f_-(z, x) B^{-1} = \fC f_-(\bar{z}, \zeta(x)) \fC^*$
for $z \in \T$.
By the uniqueness of analytic continuation, we obtain
\begin{equation*}
	 B f_+^e(z, x) = \fC f_+^e(\bar{z}, \zeta(x)) \fC^* \ \ \text{for} \ z \in \T \cup D_+,
\end{equation*}
\begin{equation*}
	f_-^e(z, x)B^{-1} = \fC f^e_-(\bar{z}, \zeta(x)) \fC^* \ \ \text{for} \ z \in \T \cup D_-.
\end{equation*}
Therefore, for $z \in \T \cup D_+ = \D^2$,
\begin{align*}
	f^e(z, x) &= f_-^e(\bar{z}^{-1}, x) f_+^e(z, x) 
		= f_-^e(\bar{z}^{-1}, x) B^{-1} \cdot B f_+^e(z, x) \\
		&= \fC f_-^e(z^{-1}, \zeta(x)) \fC^* \cdot \fC f_+^e(\bar{z}, \zeta(x)) \fC^*
		= \fC f^e(\bar{z}, \zeta(x)) \fC^*.
\end{align*}

(ii) We next consider the case of $\mathfrak{I} = (GL_n(\C), \Ad_\fC \circ *)$.
By assumption, we have
\begin{equation}\label{factnreal2}
	 f_-(z,x ) f_+(z, x) = f(z, x) = \fC f_+(\bar{z}, \zeta(x))^* \fC^* \cdot \fC f_-(\bar{z}, \zeta(x))^* \fC^*
\end{equation}
for $(z,x) \in \T \times X$.
The matrix function $\fC f_+^e(z^{-1}, \zeta(x))^* \fC^*$ (resp. $\fC f_-^e(z^{-1}, \zeta(x))^* \fC^*$) provides a continuous extension of $\fC f_+(\bar{z}, \zeta(x))^* \fC^*$ (resp. $\fC f_-(\bar{z}, \zeta(x)^* \fC^*$) onto $\T \cup D_-$ (resp. $\T \cup D_+$), which is analytic on $D_-$ (resp. $D_+$), and the right hand side of equation $(\ref{factnreal2})$ is also a canonical factorization of $f$.
Therefore, as in the proof of Lemma~\ref{frlemma1}, there exists $B \in GL_n(\C)$ satisfying,
\begin{equation*}
	B f_+^e(z, x) = \fC f^e_-(z^{-1}, \zeta(x))^* \fC^* \ \ \text{for} \ z \in \T \cup D_+,
\end{equation*}
\begin{equation*}
	f_-^e(z,x) B^{-1} = \fC f_+^e(z^{-1}, \zeta(x))^* \fC^* \ \ \text{for} \ z \in \T \cup D_-.
\end{equation*}
By these equations, for $z \in \T \cup D_+ = \D^2$,
\begin{align*}
	f^e(z, x) &= f_-^e(\bar{z}^{-1}, x) B^{-1} \cdot B f_+^e(z, x)
		= \fC f_+^e(\bar{z}, \zeta(x))^* \fC^* \cdot \fC f^e_-(z^{-1}, \zeta(x))^* \fC^*\\
		&= \fC (f_-^e(z^{-1}, \zeta(x)) f_+^e(\bar{z}, \zeta(x)))^* \fC^*
		= \fC f^e(\bar{z}, \zeta(x))^* \fC^*.
\end{align*}
\end{proof}

\begin{lemma}\label{frlemma3}
Let $\mathfrak{I}$ be a $\Z_2$-space {\upshape (i)} $(GL_n(\C)^\sa, \Ad_\fC)$ or {\upshape (ii)} $(GL_n(\C)^\sa, -\Ad_\fC)$.
Let $(X, \zeta)$ be a $\Z_2$-space and
$f \colon (\T \times X, c \times \zeta) \to \mathfrak{I}$
be a $\Z_2$-equivariant continuous map such that for each $x \in X$, $f(x)$ is a rational matrix function of trivial partial indices.
Then, its extension $f^e$ in Lemma~\ref{ffh} is a $\Z_2$-map
$f^e \colon (\D^2 \times X, c \times \zeta) \to \mathfrak{I}$.
\end{lemma}
\begin{proof}
{\upshape (i)} follows from the hermitian case of Lemma~\ref{ffh} and (i) of Lemma~\ref{frlemma1}.
For {\upshape (ii)}, note that a $\Z_2$-map
$f \colon (\T \times X, c \times \zeta) \to (GL_n(\C)^\sa, -\Ad_\fC)$
provides, by multiplication by the imaginary unit, a $\Z_2$-map
$\sqrt{-1}f \colon (\T \times X, c \times \zeta) \to (GL_n(\C)^\ska, \Ad_\fC)$.
Therefore, the result follows from the skew-hermitian case of Lemma~\ref{ffh} and (i) of Lemma~\ref{frlemma1}.
\end{proof}

Let $\nu$ be an involution\footnote{The $\Z_2$-space $(\tilde{\SSS}^3, \nu)$ is $\Z_2$-equivariantly homeomorphic to the $\Z_2$-space $\SSS^{2,2}$ in \cite{At66}.} on $\tilde{\SSS}^3$ given by the restriction of $c^2 = c \times c$ on $\C^2$ onto $\tilde{\SSS}^3$.
By Proposition~\ref{ext3} and Lemmas~\ref{frlemma1} and \ref{frlemma3}, we obtain the following result.

\begin{proposition}\label{frlemma4}
	Let $(X, \zeta)$ be a $\Z_2$-space, and let $\mathfrak{I}$ be a $\Z_2$-space $(GL_n(\C), \Ad_\fC)$, $(GL_n(\C), \Ad_\fC \circ *)$, $(GL_n(\C)^\sa, \Ad_\fC)$ or $(GL_n(\C)^\sa, -\Ad_\fC)$.
	Let $f \colon (\T^2 \times X, c^2 \times \zeta) \to \mathfrak{I}$
be a $\Z_2$-continuous map such that for each $x \in X$, $f(x)$ is a two-variable rational matrix function for which the associated quarter-plane Toeplitz operator $T^{0,\infty}_{f(x)}$ is Fredholm.
Then, through matrix factorization, there canonically associates a $\Z_2$-continuous map
$f^E \colon (\tilde{\SSS}^3 \times X, \nu \times \zeta) \to \mathfrak{I}$ that extends $f$.
\end{proposition}

\subsection{Index Theorem: Real Cases}
\label{Sect.5.2}
Let $\fR$ be the antiunitary operator on $\C^n$ given by $\fR = \diag(c,\cdots,c)$, where $c$ is the complex conjugation on $\C$.
Let $j$ be an antiunitary operator on $\C^2$ given by $j(x,y) = (-\bar{y}, \bar{x})$.
When $n$ is even, let $\fJ = \diag(j,\cdots,j)$ be a quaternionic structure on $\C^n$.
Let $A$ be a unital \CA \ equipped with a {\em real structure}\footnote{An antilinear $*$-automorphism on $A$ satisfying $\fr^2=1$.} $\fr$.
Let $\tau$ be the antiautomorphism on $A$ of order two given by $\tau(a) = \fr(a^*)$.
We call $\tau$ the {\em transposition} and write $a^\tau$ for $\tau(a)$.
The pair $(A,\tau)$ is called {\em \TA} \ in \cite{BL16}.
The transposition $\tau$ on $A$ is extended to the transposition on the matrix algebra $M_n(A)$ by $(a_{ij})^\tau = (a_{ji}^\tau)$.
Let $\sharp \bm{\otimes} \tau$ be a transposition on $M_2(A)$ defined by
\begin{equation*}
\left(
    \begin{array}{cc}
           a_{11}&a_{12}\\
           a_{21}&a_{22}
    \end{array}
\right)^{\sharp \bm{\otimes} \tau}
	=
\left(
    \begin{array}{cc}
           a_{22}^\tau & -a_{12}^\tau \\
           -a_{21}^\tau & a_{11}^\tau
    \end{array}
\right),
\end{equation*}
which is extended to the transposition on $M_{2n}(A)$ by $(b_{ij})^{\sharp \bm{\otimes} \tau} = (b_{ji}^{\sharp \bm{\otimes} \tau})$ where $b_{ij} \in M_2(A)$ and $1 \leq i,j \leq n$.
For $i = -1,0, \ldots, 6$, let $n_{r,i}$ be a positive integer, $R^{(i)}$ be a relation and $I^{(i)}_r$ be a matrix as indicated in Table~\ref{table1}.
Let $GL_n^{(i)}(A, \tau)$ be the set of all invertible elements in $M_{n_{r,i} \cdot n}(A)$ satisfying the relation $R^{(i)}$.
Following Boersema--Loring \cite{BL16}, we define\footnote{In \cite{BL16}, $KO$-groups are defined via unitaries though we can also define $KO$-groups through invertibles preserving symmetries in \cite{BL16} since deformation retraction from invertibles to unitaries preserves these symmetries. We discuss invertible elements since, in our application discussed in Sect.~\ref{Sect.6}, Hamiltonians will be expressed as multivariable nonsingular rational matrix functions, though not necessarily unitaries.}
the $KO$-group of $(A, \tau)$ as $KO_i(A, \tau) = \cup_{n=1}^\infty GL_n^{(i)}(A, \tau)/\sim_i$ where the equivalence relation $\sim_i$ is generated by homotopy and stabilization by $I^{(i)}_r$.
For a finite $\Z_2$-CW complex $(X,\zeta)$, we associate a (complex) \CA \ $C(X)$ with a transposition $\tau_\zeta$ given by $(f^{\tau_\zeta})(x) = f(\zeta(x))$.
We define $KR^{-i}(X, \zeta) = KO_i(C(X), \tau_\zeta)$ which is Atiyah's Real $K$-groups \cite{At66} for the $\Z_2$-space $(X,\zeta)$.
Note that an element $f \in GL_n^{(i)}(C(X), \tau_\zeta)$ corresponds to a $\Z_2$-equivariant continuous map $f \colon (X, \zeta) \to GL^{(i)}_{r,n}(\C)$ where $GL^{(i)}_{r,n}(\C)$ is the $\Z_2$-space as indicated in Table~\ref{table1}.
\begin{table}
\caption{$KO$-theory via invertible elements (Boersema--Loring \cite{BL16})}
\label{table1}
\centering
\begin{tabular}{|c|c|c|c|c|}
\hline
$i$ & $n_{r,i}$ & $R^{(i)}$ & $I^{(i)}_r$ & $GL^{(i)}_{r,n}(\C)$   \\ \hline \hline
$-1$ & $1$ & $x^\tau = x$ & $1$ & $(GL_n(\C), \Ad_\fR \circ*)$  \\ \hline
$0$ & $2$ & $x = x^*$, $x^\tau = x^*$ & $\diag(1, -1)$ & $(GL_{2n}(\C)^\sa, \Ad_\fR)$  \\ \hline
$1$ & $1$ & $x^\tau = x^*$ & $1$ & $(GL_n(\C), \Ad_\fR)$ \\ \hline
$2$ & $2$ & $x = x^*$, $x^\tau = -x$ & $\left( \begin{array}{cc} 0& \sqrt{-1} \\ -\sqrt{-1} & 0 \end{array} \right)$ & $(GL_{2n}(\C)^\sa, -\Ad_\fR)$ \\ \hline
$3$ & $2$ & $x^{\sharp \bm{\otimes} \tau} = x$ & $1_2$ & $(GL_{2n}(\C), \Ad_\fJ \circ*)$ \\ \hline
$4$ & $4$ & $x = x^*$, $x^{\sharp \bm{\otimes} \tau} = x^*$ & $\diag(1_2, -1_2)$ & $(GL_{4n}(\C)^\sa, \Ad_\fJ)$ \\ \hline
$5$ & $2$ & $x^{\sharp \bm{\otimes} \tau} = x^*$ & $1_2$ & $(GL_{2n}(\C), \Ad_\fJ)$ \\ \hline
$6$ & $2$ & $x = x^*$, $x^{\sharp \bm{\otimes} \tau} = -x$ & $\left( \begin{array}{cc} 0& \sqrt{-1} \\ -\sqrt{-1} &0 \end{array} \right)$ & $(GL_{2n}(\C)^\sa, -\Ad_\fJ)$ \\ \hline
\end{tabular}
\end{table}

On $l^2(\Z^2)$, we consider an antiunitary operator of order two given by the pointwise operation of complex conjugation, for which we simply write $c$.
Conjugation of $c$ provide real structures for the quarter-plane Toeplitz algebra $\TTzi$ and the pull-back \CA \ $\Szi$.
We write $\tau_{0,\infty}$ and $\tau_\mathcal{S}$ for corresponding transpositions on $\TTzi$ and $\Szi$, respectively.
For $i = -1,0, \ldots, 6$, and a positive integer $N$, let $GL_{r,N}^{(i)}$ be the $\Z_2$-space indicated in Table~\ref{table1}.
Let $(X, \zeta)$ be a finite $\Z_2$-CW complex, and let $f \colon (\T^2 \times X, c^2 \times \zeta) \to GL^{(i)}_{r,N}(\C)$ be a $\Z_2$-map such that for each $x \in X$, $f(x)$ is a two-variable rational matrix function and the associated quarter-plane Toeplitz operator $T^{0,\infty}_{f(x)}$ is Fredholm.
In this case, half-plane Toeplitz operators $\{T^0_{f(x)}\}_{x \in X}$ and $\{T^\infty_{f(x)}\}_{x \in X}$ and quarter-plane Toeplitz operators $\{T^{0,\infty}_{f(x)}\}_{x \in X}$ preserve some symmetry corresponding to the $\Z_2$-equivalence of the map $f$.
As in \cite{Hayashi3}, pairs of invertible half-plane Toeplitz operators $(T^0_f, T^\infty_f)$ defines an element $[(T^0_f, T^\infty_f)]$ of the $KO$-group $KO_i(\Szi \bm{\otimes} C(X), \tau_\mathcal{S} \bm{\otimes} \tau_\zeta)$, and (self-adjoint, when $i$ is even) Fredholm quarter-plane Toeplitz operators define an element $[T^{0,\infty}_{f}]$ of the $KR$-group $KR^{-i+1}(X, \zeta)$.
By Proposition~\ref{frlemma4}, through the matrix factorization, there associates a $\Z_2$-map $f^E$ which defines an element $[f^E]$ of the $KR$-group $KR^{-i}(\tilde{\SSS}^3 \times X, \nu \times \zeta)$.
Let
\begin{equation*}
\partial^{\text{qT}} \colon KO_i(\Szi \bm{\otimes} C(X), \tau_\mathcal{S} \bm{\otimes} \tau_\zeta) \to KO_{i-1}(C(X), \tau_\zeta) \cong KR^{-i+1}(X, \zeta),
\end{equation*}
be the boundary map of the $24$-term exact sequence for $KO$-theory associated with the short exact sequence of \TA s,
\begin{equation}\label{exttens2}
	0 \to (\K \bm{\otimes} C(X), \tau_\K \bm{\otimes} \tau_\zeta) \to (\TTzi \bm{\otimes} C(X), \tau_{0,\infty} \bm{\otimes} \tau_\zeta) \overset{\gamma \bm{\otimes} 1}{\longrightarrow} (\Szi \bm{\otimes} C(X), \tau_\mathcal{S} \bm{\otimes} \tau_\zeta) \to 0.
\end{equation}
Then, we have $\partial^{\text{qT}}([(T^0_f, T^\infty_f)]) = [T^{0,\infty}_{f}]$.
For $(\tilde{\SSS}^3, \nu)$, we take a $\Z_2$-fixed point $s_0 = (1,1)$ in $\tilde{\SSS}^3$ as its base point, and obtain an isomorphism
$KR^{-i}(\tilde{\SSS}^3 \times X, \nu \times \zeta) \cong KR^{-i}(X, \zeta) \oplus KR^{-i+1}(X, \zeta)$.
Let $\beta \colon KR^{-i}(\tilde{\SSS}^3 \times X, \nu \times \zeta) \to KR^{-i+1}(X, \zeta)$ be the projection corresponding to this decomposition.

\begin{theorem}\label{mainfamilyKO}
Let $(X, \zeta)$ be a finite $\Z_2$-CW complex.
Let $f \colon (\T^2 \times X, c^2 \times \zeta) \to GL^{(i)}_{r,N}(\C)$ be a $\Z_2$-map such that for each $x \in X$, $f(x)$ is a two-variable rational matrix function and the associated quarter-plane Toeplitz operator $T^{0,\infty}_{f(x)}$ is Fredholm.
Let $f^E \colon (\tilde{\SSS}^3 \times X, \nu \times \zeta) \to GL^{(i)}_{r,N}(\C)$ be the extension of $f$ through matrix factorization in Proposition~\ref{frlemma4}.
Then, we have $[T^{0,\infty}_f] = \beta([f^E])$ in $KR^{-i+1}(X, \zeta)$.
\end{theorem}

The proof of Theorem~\ref{mainfamily} concerns three parts: matrix factorizations, the homotopy lifting property and $K$-theory.
For Theorem~\ref{mainfamilyKO}, matrix factorizations are discussed in Sect.~\ref{Sect.5.1}.
Here, we note the following result concerning the $\Z_2$-equivariant homotopy lifting property.
For $\fC = \fR$ or $\fJ$,
let $\fc_\TT$ be an involution on $M_n(\TT)$ given by $\fc_\TT(T) = \fC T \fC^*$, and let $\fc_\T$ be an involution on $M_n(C(\T))$ given by $\fc_\T(f)(z) = \fC f(\bar{z}) \fC^*$ for $z \in \T$.

\begin{proposition}\label{refib}
The following are Serre $\Z_2$-fibrations:
\begin{enumerate}
\renewcommand{\labelenumi}{(\roman{enumi})}
\item $\sigma \colon (U_n(\TT), \fc_\TT) \to (U_n(C(\T)), \fc_\T)$,
\item $\sigma \colon (U_n(\TT), \fc_\TT \circ *) \to (U_n(C(\T)), \fc_\T \circ *)$,
\item $\sigma \colon (U_n(\TT)^\sa, \fc_\TT) \to (U_n(C(\T))^\sa, \fc_\T)$,
\item $\sigma \colon (U_n(\TT)^\sa, -\fc_\TT) \to (U_n(C(\T))^\sa, -\fc_\T)$.
\end{enumerate}
\end{proposition}

Since (\ref{fib1}) and (\ref{fib2}) are Hurewicz fibrations and by Theorem~$4.1$ of \cite{Bre67}, to show Proposition~\ref{refib}, it is sufficient to show that the restrictions on the $\Z_2$-fixed point sets are fibrations.

\begin{lemma}\label{refib2}
The following are Hurewicz fibrations:
\begin{enumerate}
\renewcommand{\labelenumi}{(\roman{enumi})}
\item $\sigma \colon U_n(\TT)^{\fc_\TT} \to U_n(C(\T))^{\fc_\T}$,
\item $\sigma \colon U_n(\TT)^{\fc_\TT \circ *} \to U_n(C(\T))^{\fc_\T \circ *}$,
\item $\sigma \colon (U_n(\TT)^\sa)^{\fc_\TT} \to (U_n(C(\T))^\sa)^{\fc_\T}$,
\item $\sigma \colon (U_n(\TT)^\sa)^{-\fc_\TT} \to (U_n(C(\T))^\sa)^{-\fc_\T}$.
\end{enumerate}
\end{lemma}
\begin{proof}
(i) and (iii) are real analogues of (\ref{fib1}) and (\ref{fib2}).
(iv) follows from the skew-adjoint analogue of (iii), since multiplication by $\sqrt{-1}$ provides the homeomorphism
from $(U_n(\TT)^\sa)^{-\fc_\TT}$ to $(U_n(\TT)^\ska)^{\fc_\TT}$.

For (ii), we put unitaries preserving these real structures in the framework of Wood \cite{Wood}.
Let $A = M_{2n}(\TT)$ and set $U(A) = U_{2n}(\TT)$.
Let
\begin{equation*}
	e = 
	\left( 
    \begin{array}{cc}
           \sqrt{-1} \cdot 1_n&0\\
           0&-\sqrt{-1} \cdot 1_n
    \end{array}
\right), \ \
	\tilde{\fC} = 
	\left( 
    \begin{array}{cc}
           0& -\fC \cdot 1_n \\
           \fC \cdot 1_n &0
    \end{array}
\right),
\end{equation*}
and let $\tilde{\fc}_\TT \colon A \to A$ be a real linear automorphism of order two given by $\tilde{\fc}_\TT(a) = \tilde{\fC} a \tilde{\fC}^*$.
Its fixed point set $A^{\tilde{\fc}_\TT}$ is a unital real Banach $*$-algebra containing $e$.
Let $\fe$ be an operator on $A^{\tilde{\fc}_\TT}$ given by the conjugation of $e$.
Let us consider the space $(U(A^{\tilde{\fc}_\TT})^{\ska})^{-\fe}$, that is, the elements in $A = M_{2n}(\TT)$ which are skew-adjoint unitary, commute with $\tilde{\fC}$ and anti-commute with $e$.
We have the following identification,
\begin{equation*}
	U_n(\TT)^{\fc_\TT \circ *} \overset{\cong}{\longrightarrow} (U(A^{\tilde{\fc}_\TT})^{\ska})^{-\fe}, \ \ a \mapsto
\left( 
    \begin{array}{cc}
           0& -a^* \\
           a&0
    \end{array}
\right).
\end{equation*}
Let $u \in U_n(\TT)^{\fc_\TT \circ *}$ and
$s = \begin{pmatrix}
0 & -u^*\\
u & 0
\end{pmatrix} \in (U(A^{\tilde{\fc}_\TT})^{\ska})^{-\fe}$.
We consider an operator $\fs$ on $U(A^{\tilde{\fc}_\TT})^{\fe}$ given by the conjugation of $k$.
By Lemma~$4.2$ of \cite{Wood}, we have the following homeomorphism,
\begin{equation*}
	\bigl( U(A^{\tilde{\fc}_\TT})^{\fe} / (U(A^{\tilde{\fc}_\TT})^{\fe})^\fs \bigl)_{1} \overset{\cong}{\longrightarrow} (U(A^{\tilde{\fc}_\TT})^{\ska})^{-\fe}_s \cong U_n(\TT)^{\fc_\TT \circ *}_u,
\end{equation*}
given by $[T] \mapsto TsT^*$.
We also have a similar homeomorphism for the algebra $C(\T)$, and as in the proof of Lemma~\ref{fiblem}, we obtain that the map
\begin{equation*}
	\sigma \colon U_n(\TT)^{\fc_\TT \circ *}_u \to U_n(C(\T))^{\fc_\T \circ *}_{\sigma(u)}
\end{equation*}
is a fiber bundle and the result follows.
\end{proof}

For the proof of Theorem~\ref{mainfamilyKO}, we replace $n_i$, $I^{(i)}$ and $GL_n^{(i)}(\C)$ in the proof of Theorem~\ref{mainfamily} with $n_{r,i}$, $I^{(i)}_r$ and $GL_{r,n}^{(i)}(\C)$ in Table~\ref{table1} and obtain the following commutative diagram:
\begin{equation*}
\vcenter{
\xymatrix{
KO_i(\Szi \bm{\otimes} C(X), \tau_\mathcal{S} \bm{\otimes} \tau_\zeta) \ar[d]_{\psi} \ar[rrd]^{\partial^{\text{qT}}} & &\\
KR^{-i+1}\bigl(([-2,2],\{ \pm 2 \}) \times \T \times X, \id \times c \times \zeta \bigl) \bigl/\Ker(\alpha) \ar[rr]^{\hspace{3cm}\bar{\alpha}} & & KR^{-i+1}(X, \zeta) \\
KR^{-i}(\tilde{\SSS}^3 \times X, \nu \times \zeta) \ar[u]^{\phi} \ar[rru]_{\beta} & &
}}
\end{equation*}
where the map $\bar{\alpha}$ is induced from the map $\alpha$, defined as the map that makes the following diagram commutative:
\begin{equation*}
\vcenter{
\xymatrix{
KR^{-i}(\T \times X, c \times \zeta) \oplus KR^{-i}(\T \times X, c \times \zeta) \ar[d]_{\partial^{\text{pair}}} \ar[rd]^{\hspace{3mm} \partial^{\text{T}} \oplus - \partial^{\text{T}}}& \\
KR^{-i+1}\bigl(([-2,2],\{ \pm 2 \}) \times \T \times X, \id \times c \times \zeta \bigl) \ar[r]^{\hspace{2cm}\alpha} & KR^{-i+1}(X, \zeta)
}}
\end{equation*}
Theorem~\ref{mainfamilyKO} is proved in a parallel way as Theorem~\ref{mainfamily} by using real $K$-theory in place of complex $K$-theory.

\section{Gapped Topological Invariants Related to Corner States}
\label{Sect.6}
In this section, we discuss some bulk-edge gapped Hamiltonians on a lattice with a codimension-two corner or hinge.
By using Theorem~\ref{mainfamily} and Theorem~\ref{mainfamilyKO}, we provide more geometric way to formulate a relation between topological invariants for such gapped Hamiltonians and corner/hinge states than that in \cite{Hayashi2,Hayashi4}.

We consider a translation invariant Hamiltonian on the lattice $\Z^n$ of the following form:
\begin{equation*}
	H \colon l^2(\Z^n, \C^N) \to l^2(\Z^n, \C^N), \ \ H = \sum_{\text{finite}} a_{j_1 \cdots j_n} S_1^{j_1} \cdots S_n^{j_n},
\end{equation*}
where $a_{j_1 \cdots j_n} \in M_N(\C)$, $S_j$ is the shift operator in the $j$-th direction and the subscript {\em finite} means that $a_{j_1 \cdots j_n} = 0$ except for finitely many $(j_1, \ldots, j_n) \in \Z^n$.
We assume that $H$ is self-adjoint and consider such Hamiltonians in each of the ten Altland--Zirnbauer classes.
In classes A, AI and AII, we further assume that the spectrum of the bulk Hamiltonian is not contained in $\R_{>0}$ and $\R_{<0}$.
Through the Fourier transform, the bulk Hamiltonian corresponds to a hermitian matrix-valued function $H \colon \T^n \to M_N(\C)^\sa$ on the $n$-dimensional Brillouin torus $\T^n$.
We write $(z_1, \ldots, z_n)$ for an element in $\T^n$.
Corresponding to our finite hopping range condition, each entry of this matrix consists of a $n$-variable Laurent polynomial; therefore, the bulk Hamiltonian $H$ correspond to a $n$-variable rational matrix function on $\T^n$.
We next introduce our models for two edges and the corner.
Note that, for each $\bz = (z_3, \ldots, z_n) \in \T^{n-2}$, the matrix function $H(\bz)$ on $\T^2$ is a two-variable rational matrix function.
Let $H^0(\bz) = T^0_{H(\bz)}$ and $H^\infty(\bz) = T^\infty_{H(\bz)}$ be the associated half-plane Toeplitz operators, and let $H^{0,\infty}(\bz) = T^{0,\infty}_{H(\bz)}$ be the associated quarter-plane Toeplitz operator.
That is, for models of two edges, we consider the restrictions of our bulk Hamiltonian onto half-spaces $\Z \times \Z_{\geq 0} \times \Z^{n-2}$ and $\Z_{\geq 0} \times \Z \times \Z^{n-2}$, and for the model of codimension-two right angle corner, we consider the restriction onto the lattice $(\Z_{\geq 0})^2 \times \Z^{n-2}$, where we assume the Dirichlet boundary condition.
We assume that, for any $\bz \in \T^{n-2}$, half-plane Toeplitz operators $H^0(\bz)$ and $H^\infty(\bz)$ are invertible.
Under this assumption, $H(\bz)$ is also invertible and $H^{0,\infty}(\bz)$ is Fredholm.
Therefore, we assume that our model Hamiltonians for the bulk and two edges that makes the corner are gapped.
Under this assumption, we discuss a relation between some gapped topological invariant and corner states.
As in \cite{Hayashi4}, the family of self-adjoint Fredholm operators $\{ H^{0,\infty}(\bz) \}_{\bz \in \T^{n-2}}$ defines an element of the complex $K$-group $K^{-i+1}(\T^{n-2})$ for classes A and AIII, or the $KR$-group $KR^{-i+1}(\T^{n-2}, c^{n-2})$ for classes AI, BDI, D, DIII, AII, CII, C and CI of some degree $i$ corresponding to its Altland--Zirnbauer class $\spadesuit$ as indicated in Table~\ref{label}.
\begin{table}
\caption{$i$ and $\mathcal{H}^E$ for each of the Altland--Zirnbauer classes $\spadesuit$}
\label{label}
\centering
\begin{tabular}{|c||c|c|c|c|c|c|c|c|c|c|} \hline
  $\spadesuit$ & $\A$ & $\AIII$ & $\AI$ & $\BDI$ & $\DD$ & $\DIII$ & $\AII$ & $\CII$ & $\CC$ & $\CI$ \\ \hline \hline
  $i$ & $0$ & $1$ & $0$ & $1$ & $2$ & $3$ & $4$ & $5$ & $6$ & $-1$ \\ \hline
  $\mathcal{H}^E$ & $H^E \oplus 1_N$ & $h^E$ & $H^E \oplus 1_N$ & $h^E$ & $H^E$ & $h^E$ & $H^E \oplus 1_N$ & $h^E$ & $H^E$ & $h^E$ \\ \hline
\end{tabular}
\end{table}
We write $\I_{\mathrm{Gapless}}^{\spadesuit}(H)$ for this element of the $K$-group.
If $\I_{\mathrm{Gapless}}^{\spadesuit}(H)$ is non-trivial, there exist topological corner/hinge states.
For classes AIII, BDI, DIII, CII and CI where the Hamiltonians preserve chiral symmetry, a Hamiltonian $H$ anti-commute pointwise with the chiral symmetry operator $\Pi$ and can be represented by the off-diagonal form
$H = \begin{pmatrix}
0 & h^*\\
h & 0
\end{pmatrix}$.
This $h$ is also a nonsingular $n$-variable rational matrix function on $\T^n$.
As in \cite{Hayashi4}, this $H$ or $h$ preserves the symmetries of Boersema--Loring's formulation of complex or real $K$-theory groups.
Under our assumption, by using matrix factorizations (Proposition~\ref{exthf} and Proposition~\ref{frlemma4}), $H$ or $h$ on $\T^n$ is extended to a nonsingular matrix-valued continuous map
$H^E \colon \tilde{\SSS}^3 \times \T^{n-2} \to GL_N(\C)^\sa$ or $h^E \colon \tilde{\SSS}^3 \times \T^{n-2} \to GL_\frac{N}{2}(\C)$.
Corresponding to its Altland--Zirnbauer class, this matrix function define the following element of the complex or real $K$-group which is a gapped topological invariants for our bulk-edge gapped system:

\begin{definition}\label{def6.1}
In classes $\spadesuit=$ A and AIII, we define,
\begin{equation*}
	\I_{\mathrm{Gapped}}^{\spadesuit}(H) = [\mathcal{H}^E] \in K^{-i}(\tilde{\SSS}^3 \times \T^{n-2}).
\end{equation*}
In classes $\spadesuit=$  AI, BDI, D, DIII, AII, CII, C and CI, we define,
\begin{equation*}
	\I_{\mathrm{Gapped}}^{\spadesuit}(H) = [\mathcal{H}^E] \in KR^{-i}(\tilde{\SSS}^3 \times \T^{n-2}, \nu \times c^{n-2}),
\end{equation*}
where $i$ and $\mathcal{H}^E$ are as indicated in Table~\ref{label}.
\end{definition}

As in \cite{Hayashi2,Hayashi4}, from the pair of invertible half-space operators $H^0$ and $H^\infty$, we can define an element of the $K$-group $K_i(\Szi \bm{\otimes} C(\T^{n-2}))$ or $KO_i(\Szi \bm{\otimes} C(\T^{n-2}), \tau_\mathcal{S} \bm{\otimes} \tau_{c^{n-2}})$, and $\I_{\mathrm{Gapless}}^{\spadesuit}(H)$ is the image of this element through the boundary map $\partial^{\text{qT}}$ of the long exact sequence of $K$-theory for \CA s or $KO$-theory for \TA s associated with the extension (\ref{exttens1}) or (\ref{exttens2}).
Therefore, by Theorem~\ref{mainfamily} and Theorem~\ref{mainfamilyKO}, we obtain the following result.
\begin{theorem}\label{thm6.2}
In classes $\spadesuit= \A$ and $\AIII$, we have
\begin{equation*}
	\beta(\I_{\mathrm{Gapped}}^{\spadesuit}(H)) = \I_{\mathrm{Gapless}}^{\spadesuit}(H) \in K^{-i+1}(\T^{n-2}),
\end{equation*}
In classes $\spadesuit= \AI, \BDI, \DD, \DIII, \AII, \CII, \CC$ and $\CI$, we have,
\begin{equation*}
	\beta(\I_{\mathrm{Gapped}}^{\spadesuit}(H)) = \I_{\mathrm{Gapless}}^{\spadesuit}(H) \in KR^{-i+1}(\T^{n-2}, c^{n-2}),
\end{equation*}
where $i$ and $\mathcal{H}^E$ are as indicated in Table~\ref{label}.
\end{theorem}

Theorem~\ref{thm6.2} provides a geometric formulation for a relation between a gapped topological invariant and corner states in \cite{Hayashi2,Hayashi4},
though, since our three sphere $\tilde{\SSS}^3$ is not smooth as the boundary of $\D^2 \times \D^2$,
an integration formula for numerical gapped invariants, like integration of the Berry curvature for the first Chern number, is still missing.
At this stage, we simply note the following understanding of numerical gapped invariants for two-dimensional class AIII systems with a corner and three-dimensional class A systems with a hinge.

\begin{example}\label{2DAIII}
For a two-dimensional class AIII bulk-edge gapped Hamiltonian on the lattice $\Z_{\geq 0} \times \Z_{\geq 0}$,
our extension of the (off-diagonal part of the) bulk Hamiltonian defines an element $\I_{\mathrm{Gapped}}^{\AIII}(H) = [h^E] \in K^{-1}(\tilde{\SSS}^3)$.
By Corollary~\ref{maincorollary}, the three-dimensional winding number $W_3(h^E)$ of $h^E$ is the same as the Fredholm index $\ind T^{0,\infty}_h = \mathrm{Tr}(\Pi \lvert_{\Ker H^{0, \infty}}) = \I_{\mathrm{Gapless}}^{\AIII}(H)$ and accounts for topological corner states.
The two-variable rational matrix function $f$ in Example~\ref{example} provides an example in this class and corresponds to Benalcazar--Bernevig--Hughes' two-dimensional model of a second-order topological insulator \cite{BBH17a} as discussed in \cite{Hayashi3}.
\end{example}

\begin{example}\label{3DA}
For a three-dimensional class A bulk-edge gapped Hamiltonian on the lattice $\Z_{\geq 0} \times \Z_{\geq 0} \times \Z$,
our gapped topological invariant is $\I_{\mathrm{Gapped}}^{\A}(H) = [H^E \oplus 1_N] \in K^0(\tilde{\SSS}^3 \times \T)$.
Since $H^E$ is a continuous family of self-adjoint invertible matrices, we define a complex vector bundle $E$ on  $\tilde{\SSS}^3 \times \T$ whose fiber at $(z,w,t) \in \tilde{\SSS}^3 \times \T$ is
\begin{equation*}
	 E_{(z,w,t)} = \bigcup_{\mu < 0} \Ker (H^E(z,w,t) - \mu).
\end{equation*}
This vector bundle $E$ is an extension of the Bloch bundle since $H^E$ is an extension of the bulk Hamiltonian $H$.
By Theorem~\ref{mainfamily} and the results in \cite{Hayashi2}, the minus of the pairing of second Chern class of this extended Bloch bundle $E$ with fundamental class of $\tilde{\SSS}^3 \times \T$ is the same as the spectral flow of the family of self-adjoint Fredholm operators $\{ H^{0,\infty}(t) \}_{t \in \T}$, therefore accounts for the number of topological hinge states\footnote{For its proof, we take a deformation of our three sphere $\tilde{\SSS}^3$ to the unit three sphere $\SSS^3$ in $\C^2$. For two- and four-dimensional unit balls in $\C$ and $\C^2$, we consider spin$^c$ structures they inherit as subspaces of Euclidean spaces. We equip $\T$ and $\SSS^3$ their boundary spin$^c$ structures and take the product spin$^c$ structure on $\SSS^3 \times \T$. Then the result follows from Atiyah--Singer's index formula for twisted spin$^c$ Dirac operators \cite{AS68}.}.
\end{example}

Summarizing, we consider a gapped translation invariant single-particle Hamiltonian of finite hopping range on the lattice $\Z^n$ in each of the ten Altland--Zirnbauer classes.
We use Gohberg--Kre{\u \i}n theory to factorize the bulk Hamiltonian on the Brillouin torus about two variables $z_1$ and $z_2$.
For this purpose, there is a relevant algorithm since our bulk Hamiltonian corresponds to a multivariable rational matrix function on the torus $\T^n$ \cite{GK58r, CG81, GKS03}.
If all of the partial indices of right matrix factorizations are trivial (equivalently, if compressions of our bulk Hamiltonians onto two half-spaces $\Z \times \Z_{\geq 0} \times \Z^{n-2}$ and $\Z_{\geq 0} \times \Z \times \Z^{n-2}$ are invertible), we define two topological invariants:
One is defined through the restriction of the bulk Hamiltonian onto the lattice $(\Z_{\geq 0})^2 \times \Z^{n-2}$ assuming the Dirichlet boundary condition which provides a (family of) Fredholm operator(s) and its $K$-class $\I_{\mathrm{Gapless}}^{\spadesuit}(H)$ account for topological corner/hinge states.
The other is defined as the $K$-class $\I_{\mathrm{Gapped}}^{\spadesuit}(H)$ of the extension of the bulk Hamiltonian onto $\tilde{\SSS}^3 \times \T^{n-2}$ obtained through matrix factorizations, which is our gapped topological invariant (Definition~\ref{def6.1}).
There is a relation between these two topological invariants (Theorem~\ref{thm6.2}), therefore, corresponding to the gapped topological invariant $\I_{\mathrm{Gapped}}^{\spadesuit}(H)$, corner states appear.

\subsection*{Acknowledgments}
This work was supported by JSPS KAKENHI (Grant Nos. JP17H06461, JP19K14545) and JST PRESTO (Grant No. JPMJPR19L7).


\end{document}